\documentclass[pdflatex,sn-basic,Numbered]{sn-jnl}

\usepackage{multirow}%
\usepackage{amsmath,amssymb,amsfonts}%
\usepackage{amsthm}%
\usepackage{mathrsfs}%
\usepackage[title]{appendix}%
\usepackage{xcolor}%
\usepackage{textcomp}%
\usepackage{manyfoot}%
\usepackage{booktabs}%
\usepackage{algorithm}%
\usepackage{algorithmicx}%
\usepackage{algpseudocode}%
\usepackage{listings}%
\usepackage{comment}
\usepackage{graphicx} 
\usepackage{amsmath}
\usepackage{amsfonts}
\usepackage{times} 
\usepackage{eqnarray}
\usepackage{latexsym}
\usepackage{booktabs}
\usepackage{amsfonts}
\usepackage{xcolor}
\usepackage{mathtools}
\usepackage{enumitem}
\usepackage{tikz}
\usepackage{graphicx} 
\usepackage{float} 
\usetikzlibrary{shapes,arrows}
\newtheorem{thm}{Theorem}[section]
\newtheorem{cor}[thm]{Corollary}
\newtheorem{lem}[thm]{Lemma}
\newtheorem{lemma}[thm]{Lemma}

\newtheorem{prop}[thm]{Proposition}
\theoremstyle{definition}
\newtheorem{defn}[thm]{Definition}
\newtheorem{rem}[thm]{Remark}

\newtheorem{assumption}[thm]{Assumption}
\usepackage{graphicx} 
\newcommand{\ddt}{\tfrac{\text{\normalfont d}}{\text{\normalfont d}t}}
\numberwithin{equation}{section}
\allowdisplaybreaks

\newcommand{\yref}{y_{\rm ref}}

\DeclareMathOperator{\dom}{dom}

\newcommand{\setdef}[2]{\left\{\, #1 \left|\, \vphantom{#1} #2\right.\right\}}

\renewcommand{\emptyset}{\varnothing}

\newcommand{\R}{\ensuremath{\mathbb R}}    
\newcommand{\N}{\ensuremath{\mathbb N}}    

\newcommand{\eps}{\varepsilon}

\newcommand{\AxisRotator}[1][rotate=0]{%
    \tikz [x=0.2cm,y=0.40cm,line width=.2ex,-stealth,#1] \draw (0,0) arc (-150:150:1 and 1);%
}

\theoremstyle{thmstyletwo}%

\theoremstyle{thmstylethree}%

\newcommand{\blue}[1]{#1}

\begin{document}

\title{Analysis and funnel control for nonlinear drill strings}

\author*[1]{\fnm{Thomas} \sur{Berger}}\email{thomas.berger@mathematik.uni-halle.de}
\equalcont{These authors contributed equally to this work.}

\author[2]{\fnm{Thavamani} \sur{Govindaraj}}\email{thavamani.maths@kongu.ac.in}

\author[1]{\fnm{Pushya} \sur{Mitra}}\email{pushya.mitra@mathematik.uni-halle.de}
\equalcont{These authors contributed equally to this work.}

\author[3]{\fnm{Timo}\sur{Reis}}\email{timo.reis@tu-ilmenau.de}
\equalcont{These authors contributed equally to this work.}

\affil*[1]{\orgdiv{Institut für Mathemtik}, \orgname{Martin-Luther-Universität Halle-Wittenberg}, \orgaddress{\street{Theodor-Lieser-Straße 5}, \postcode{06120} \city{Halle (Saale)},  \country{Germany}}}

\affil[2]{\orgdiv{Department of Mathematics}, \orgname{Kongu Engineering College}, \orgaddress{\street{Perundurai}, \postcode{638060} \city{Tamil Nadu}, \country{India}}}

\affil[3]{\orgdiv{Systems Theory and Partial Differential Equations Group}, \orgname{Technische Universität Ilmenau}, \orgaddress{\street{Weimarer Str. 25}, \postcode{98684} \city{Ilmenau}, \country{Germany}}}

\abstract{ We study the output tracking problem for a vertically driven drill string system described by a nonlinear boundary-coupled PDE-ODE model. Solvability analysis of the drill string model is achieved by first casting the model in an abstract boundary value problem involving set-valued operators on an appropriate Hilbert space. The governing equation here consists of evolution and the damping part. Existence of solutions is established within the framework of maximal monotone operators where one first proves that the evolution operator is a linear skew-adjoint operator and the  distributed damping term is a Nemytskii relation which is then proven to be maximal monotone. Maximal monotonicity of the combined operator is then a consequence of Rockafellar's theorem. Furthermore, we propose a novel funnel control design that ensures the angular velocity of the drill bit follows a dynamically adjusted reference trajectory, while the tracking error remains confined within a pre-specified performance funnel. The reference adjustment mechanism adapts in response to large wave traveling times that may cause performance degradation. The corresponding feasibility result is illustrated by some simulations.}

\keywords{Nonlinear PDE-ODE model, maximal monotone operator, wave equation, funnel control}

\maketitle

\section{Introduction}\label{sec1}

One major issue in industrial drilling is the stick-slip phenomenon, a form of torsional instability that degrades drilling performance, increases wear, and can cause catastrophic failures~\cite{tian2020research}. Research of this phenomenon began in the 1980s, with early studies identifying friction at the drill bit as a key factor. Over time, control mechanisms such as proportional-integral controllers and dynamic modeling have been developed to mitigate stick-slip effects, though practical applications remain challenging~\cite{TerAndVin2020}. To address this, an integrated vibration tool that combines drill string mechanics with operational constraints has been introduced, leading to more precise control over torque and angular velocity~\cite{tian2020research}. Experimental trials confirm its effectiveness in reducing stick-slip oscillations and improving drilling stability. Traditional mitigation strategies rely on torque feedback systems, but these require frequent recalibration and can sometimes amplify instability~\cite{sharma2024review}. There are alternative approaches that focus on controlling torsional wave interactions and treating the drill string as a system of interconnected pendula, using numerical models to simulate and manage vibrations~\cite{athanasiou2020simulation}. Despite progress, unpredictable drill string behavior remains a challenge, as small control parameter changes can destabilize drilling operations. Bifurcation analysis and delay differential equations have helped map stability regimes, offering insights into broader nonlinear dynamic systems~\cite{BALANOV2003}. Continued advancements in theoretical modeling, simulations, and real-world testing are crucial for improving drilling efficiency while minimizing failures. Moreover, insights from drill string dynamics have applications in various engineering fields, including power transmission and optical fiber communications~\cite{moharrami2021nonlinear}.

A drilling system, in general, consists of a flexible drill string, drill collars, a cutter at the bottom and a rotary table at the top of the drill string (see Fig.~\ref{fig.drill}). During the drilling process, the rotary table transmits torque through the drill string to drive the cutter. Owing to the inherent flexibility of the drill string, torsional, lateral, and axial vibrations may arise. The suppression of such vibrations is critical to ensure the efficiency, stability, and overall performance of the drilling system. In \cite{BALANOV2003}, torsional stability of a vertically driven drill string is addressed using numerical techniques. In \cite{CRUZNETO2023}, an optimal static output feedback controller has been proposed to suppress torsional vibrations of a nonlinear drill string model. The article \cite{SagMegKrsRou2013} addresses a stabilization problem of a drill string in oil well drilling systems. An optimal state-feedback control strategy for mitigating drill string rotational motion, based on a linearized system model, has been introduced in \cite{AndAhm2003}. Output regulation of a wave equation subject to two dynamic boundary conditions is considered in \cite{Chitour2023}, where a proportional-integral control is designed to achieve the control objective. Proportional–integral controllers for drill string models were likewise proposed in~\cite{TerAndVin2020}.

\begin{figure}[ht]
    \centering
    \includegraphics[width=0.5\linewidth]{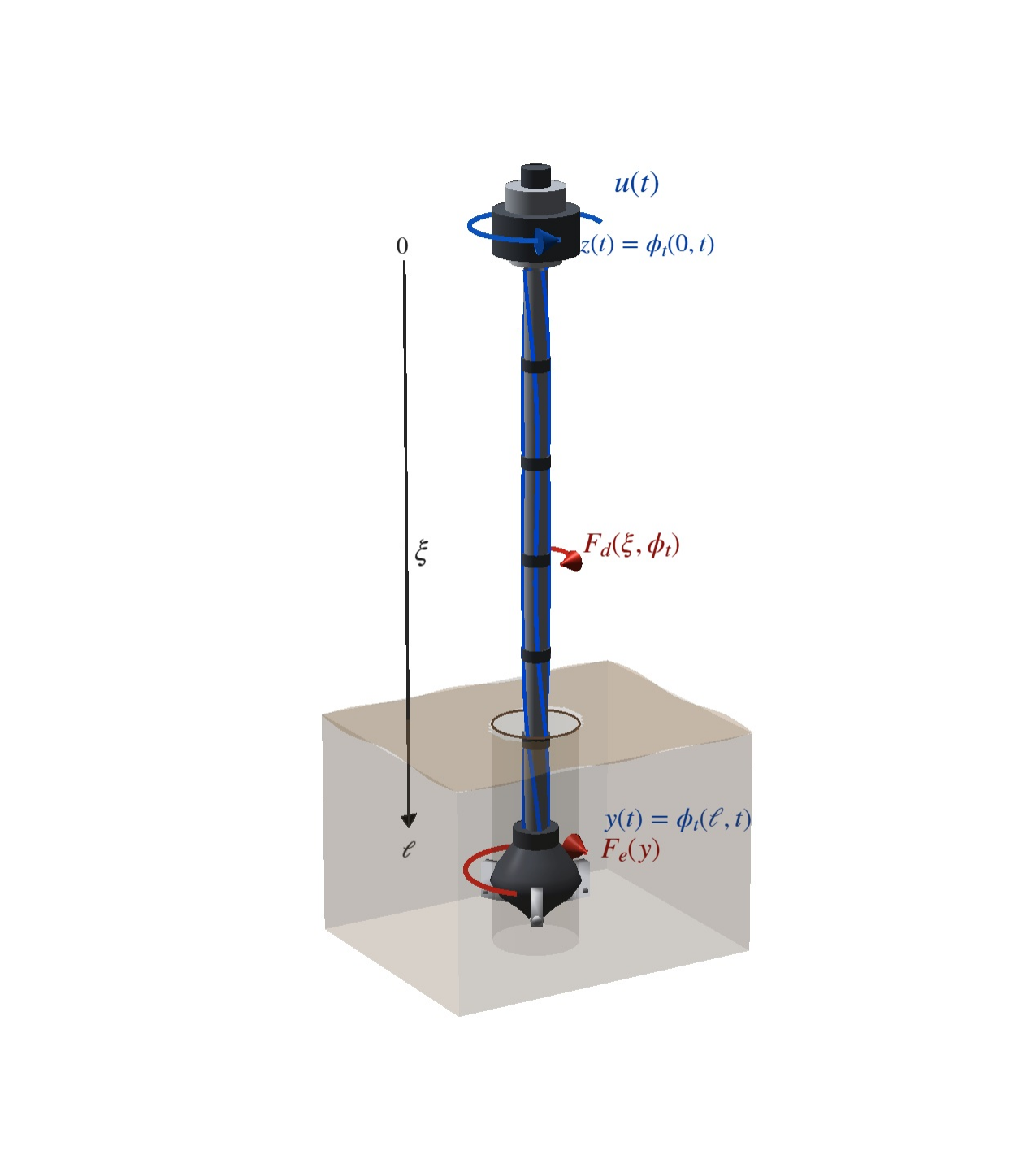}
    \caption{Drill-string depicting the torsional wave propagating upward as an effect of the collision of the bottom hole assembly (drill bit) against the surface. $z(t)$ depicts the (measured) angular velocity at the top and $y(t)$ depicts the (to be controlled) angular velocity at the bottom. }
    \label{fig.drill}
\end{figure}

In the present work, we introduce a broad class of drill string
models in which both the distributed damping and the damping at the
cutting end are modeled as set-valued maps, thereby allowing for
nonsmooth damping and friction mechanisms. Thus, the drill string can be mathematically modeled by a coupled system of nonlinear partial and ordinary differential inclusions. In the absence of axial and lateral vibrations, the torsional dynamics of the vertical drill string are described by  
\begin{equation}\label{drillstring}
\begin{aligned}
    \rho(\xi)\phi_{tt}(\xi,t)&\in(G(\xi)\phi_{\xi}(\xi,t))_{\xi}-F_d(\xi,\phi_t(\xi,t)),\quad 0<\xi<\ell,\ t>0,\\
    J\phi_{tt}(\ell,t)&\in-\Gamma (G\phi_{\xi})(\ell,t)-F_e(\phi_t(\ell,t)),\\
(G\phi_{\xi})(0,t)&=u(t),\\
    \phi(\xi,0)&=\phi_0(\xi),\quad \phi_t(\xi,0)=v_0(\xi),\\
    y(t)&=\phi_t(\ell,t),\\
    z(t)&=\phi_t(0,t),
\end{aligned}
\end{equation}
where $\phi(\xi,t)$ denotes the rotary angle of the drill string at position $\xi$ and time $t$,  
$\rho(\xi)$ and $G(\xi)$ are the mass density and shear modulus of the string, respectively,  
$F_d(\xi,\cdot)$ is a (possibly set-valued) mapping describing the distributed damping torque at position $\xi$,  
$F_e$ is a (possibly set-valued) mapping describing the concentrated damping torque acting at the cutting end,  
$J$ is the rotary inertia at $\xi=\ell$,  
$\Gamma$ is the second moment of area at $\xi=\ell$,  
$u(t)$ is the applied torque control input at the top end $\xi=0$,  
$y(t)$ is the angular velocity at the bottom end $\xi=\ell$, which serves as the controlled output,  
$z(t)$ is the angular velocity at the top end $\xi=0$, which is taken as the measured output of the model,  
and $\phi_0(\xi)$ and $v_0(\xi)$ denote the initial rotary angle and angular velocity, respectively. 

The mappings $F_d$ and $F_e$ are allowed to be set-valued in order to include nonsmooth friction laws, in particular Coulomb-type friction with sticking~\cite{Rabinowicz1956}. Set-valued formulations of dry friction are standard in nonsmooth mechanics \cite{Filippov1988DifferentialEW}, \cite{Leine2004}, and have been used to model stick-slip in drill strings, see \cite{10.1115/1.1452745}. The maximal monotonicity assumptions imposed below restrict this class to monotone friction relations; in particular, general velocity-weakening friction laws, such as Stribeck-type characteristics, are not covered by the open-loop well-posedness result. The set-valued formulation is not essential to the drill-string model itself, but allows for a unified solvability treatment of such nonsmooth damping mechanisms. Clearly, the drill string model given by \eqref{drillstring} can be viewed as a coupled PDE-ODE model with dynamic boundary condition. Similar models  are also treated in~\cite{lhachemi:hal-05038849} where the authors treat the problem of control design for classes of parabolic and hyperbolic PDEs in the presence of input saturation.  Output-feedback control with the objective of stabilizing such systems was investigated in~\cite{deutscher2018output} and further developed in~\cite{Deutscher03102019}.  Problems involving dynamic boundary conditions, where the boundary dynamics are governed by additional evolution equations (Robin type boundary conditions), have also been  investigated for example in \cite{favini2001nonlinear} for evolution equations with second order differential operator and where the related boundary condition admits a maximal monotone graph. A practically relevant system is considered in~\cite{avalos2001well} where a structural acoustics control model with point observation of the pressure is investigated. The desired goal is then achieved via techniques from microlocal analysis and the theory of pseudodifferential operators. The work~\cite{graber2012existence} also deals with a similar mathematical problem as considered in the present paper. There the authors consider a strongly damped wave equation with nonlinear dynamic boundary conditions using techniques from nonlinear semigroup theory.




The aim of the present paper is twofold. First, we establish the existence of solutions for the nonlinear drill string model~\eqref{drillstring} in open-loop configuration using the theory of nonlinear evolution equations. Second, we develop a novel funnel control design that guarantees that in the closed-loop system the cutter's angular velocity follows a dynamically adjusted reference trajectory while remaining within a prescribed performance envelope. To this end, we impose slightly stronger assumptions, such as requiring the damping effects to be modeled by conventional (single-valued) functions.
The adjustment of the reference signal mitigates performance degradation caused by large wave traveling times and the corresponding delay in the effect of the input $u(t)$ on the controlled output $y(t)$. The main novelty compared to existing results lies in addressing a nonlinear drill string model and introducing a simple funnel control strategy that guarantees prescribed performance of the dynamically adjusted tracking error.  
The concept of \emph{funnel control} was introduced in~\cite{IlcRyaSan2002} for a class of first-order nonlinear systems described by functional-differential equations. Since then, it has been extended to higher-order systems~\cite{Berger2018,BerIlcRya2021} and to coupled PDE--ODE systems~\cite{Ber2020,BergPuch22}; see also the survey~\cite{BergIlch25} for an overview. Funnel control is attractive because of its simplicity and low complexity. Moreover, in its classical form, it requires neither knowledge of the system parameters nor access to the full system state, which makes it inherently robust. Successful applications of funnel control include chemical reactor models~\cite{IlcTre2004}, wind turbine systems~\cite{Hac2014}, industrial servo systems~\cite{Hack17}, underactuated multibody systems~\cite{BergDrue21,DrueLanz24}, and electrical circuits~\cite{BerRei2014}.

\subsection{Control Objective}\label{Ssec:ContrObj}

The objective is to design an output error feedback such that, for a given reference trajectory $y_{\text{ref}} \in  W^{1,\infty}(\R_{\ge 0},\R)$, in the resulting closed-loop system the tracking error  $y(t)-y_{\text{ref}}(t)$ exhibits a prescribed performance $|y(t)-y_{\text{ref}}(t)|<\psi(t)$ for $t\ge 0$, where the function $\psi\in W^{1,\infty}(\R_{\ge 0},\R)$ is such that $\inf_{t\ge 0} \psi(t) > 0$. Furthermore, all closed-loop signals should remain bounded. The function~$\psi$ is a controller design parameter and can be chosen so that, depending on the application, a desired transient and steady-state behavior is achieved. Therefore, a typical choice is $\psi(t) = a e^{-bt} + c$ for suitable $a,b,c>0$ selected to prescribe the maximum output tracking error at steady-state and the minimum convergence rate, respectively.

It is important to note that the controller to be designed only receives the information of the measured output $z(t)$, which is the angular velocity at the top of the drill string. It uses the
length $\ell$ and the material constants $\rho$, $G$ of the string. The damping laws
$F_d$, $F_e$ and the state of the string are unknown to the controller, and the bit parameters
$J$ and $\Gamma$ enter only through the choice of the controller parameters $\alpha,\beta$ in the
feasibility result Theorem~\ref{Thm:funnel}. On the other hand, the objective is to achieve tracking for the to be controlled output $y(t)$,  which is the angular velocity at the bottom of the drill string. Since the control input is itself applied at the top of the string, it affects the output $y(t)$ only with a certain delay corresponding to the wave traveling time. Therefore, it impossible to achieve the control objective when $y_{\rm ref}$ and $\psi$ are arbitrarily prescribed and the initial conditions are unknown~-- the closed-loop system will exhibit blow-up of solutions in general~\cite{BergBika25}. To avoid this, the controller is appended by a dynamic adaptation mechanism for the reference signal which actively compensates for the delay between $u(t)$ and $y(t)$.

\subsection{Notation}
We denote by $\mathbb{R}$ the set of real numbers and $\mathbb{R}_{\geq 0}=[0, \infty)$. For a~subset $N$ of a~topological space, $\operatorname{int}N$ and $\overline{N}$ respectively stand for the interior and closure of $N$. For a Hilbert space $X$, we denote its inner product by $\langle \cdot,\cdot\rangle$, and for an interval $I \subset \mathbb{R}$, we denote by $L^p(I,X)$ the usual Lebesgue-Bochner space of Bochner measurable functions $f:I \to X$ with $\|f\|_X\in L^p(I,\R)$, where $1\leq p \leq \infty$. We denote by $W^{k,\infty}(I,X)$ the set of $k$-times weakly differentiable functions $f:I \to X$ such that $f,f',\ldots,f^{(k)} \in L^{\infty}(I, X)$ and by $C^k(I,X)$ the set of all $k$-times continuously differentiable functions $f:I\to X$, with $C(I,X):=C^0(I,X)$. 
For a (possibly nonlinear) operator $A$, we denote by $\operatorname{dom} (A)$ and $\mathcal{R}(A)$, the domain and the range of $A$, respectively. We denote by $\mathbb{P}(X)$ the power set of $X$. Given a domain $\Omega$, we denote by $C_0^\infty(\Omega,\R)$ the set of infinitely differentiable real-valued functions with compact support in $\Omega$.

\subsection{Organization of the article}
The paper is organized as follows.
In Section~\ref{Sec:Prelim}, we introduce the necessary background on operators and relations in Hilbert spaces and present new results on abstract boundary value problems involving set-valued operators. These results provide the theoretical foundation for the subsequent solvability analysis of the drill-string model, which is carried out in Section~\ref{Sec:wellpose}. Moreover, we present explicit solution formulas under some additional assumptions. Section~\ref{Sec:ContrDes} is devoted to the presentation of a novel funnel controller to achieve the control objective described in Subsection~\ref{Ssec:ContrObj}, together with a proof for its feasibility. Numerical simulations illustrate and validate the controller design in Section~\ref{Sec:Sim} and some conclusions are given in Section~\ref{Sec:Concl}.

\section{Multivalued maps and nonlinear evolution equations}\label{Sec:Prelim}
In this section, we recall basic results on maximal monotone operators. 
Throughout, let $X$ and $Y$ denote real Hilbert spaces. 
A \emph{relation} between $X$ and $Y$ is simply a subset $A \subset X \times Y$.  
For $x \in X$ we define the \emph{image} of $x$ under $A$ as
\[
    Ax := \setdef{y \in Y}{(x,y)\in A}.
\]
If $Ax \neq \emptyset$, then $x$ is said to belong to the \emph{domain} of $A$, and the union of all $Ax$ is the \emph{range} of $A$, i.e.,
\[\operatorname{dom} (A) := \setdef{x \in X}{Ax \neq \emptyset},\qquad
\mathcal{R}(A):=\bigcup_{x \in X} Ax.\]
If $\mathcal{R}(A) = Y$, then we call $A$ {\em surjective}. Furthermore, we call $A$ {\em injective}, if
\[
    \forall\, x_1, x_2\in X:\ \big( Ax_1 \cap Ax_2 \neq \emptyset\big)\ \implies\ x_1 = x_2.
\]
Every relation $A\subset X \times Y$ has an inverse, defined by the relation
\[
    A^{-1} = \setdef{(y,x)}{(x,y)\in A}\subset Y \times X.
\]
In the relational sense above, an ordinary (single-valued) function $f:X \to Y$ is simply a relation whose images are singletons. If, more generally, images $Ax$ may contain more than one element, then $A$ is called a \emph{multivalued operator} (or \emph{multivalued mapping}) from $X$ to $Y$, and we write 
$A:X \rightrightarrows Y$. Obviously, relations and multivalued operators describe the same object; the difference lies only in the chosen viewpoint or notation.
In analogy with convenient multivalued mappings, $A,B: X\rightrightarrows Y$ and $\lambda \in \R$, we define the scaled and summed relations as
\[
\lambda A = \setdef{(x,\lambda y)}{(x,y)\in A},\qquad
A+B = \setdef{(x,y+z)}{(x,y)\in A,\ (x,z)\in B}.
\]
\begin{defn}[(Maximal) monotone operators]
Let $X$ be a~real Hilbert space. A multivalued operator $A: X \rightrightarrows X$ is called \emph{monotone}, if
\[
\langle x_1 - x_2,\, y_1 - y_2 \rangle \ge 0 
    \quad \text{for all } (x_i,y_i) \in A,\ i=1,2.
\]
The operator $A$ is called \emph{maximal monotone}, if it is monotone and not properly contained in any other monotone multivalued operator.  
\end{defn}
We recall some classical results on maximal monotone operators. The next theorem is originally stated for Banach spaces; here we present a simplified version tailored to Hilbert spaces.

\begin{thm}[{Minty's theorem~\cite[Thm.~2.2]{Bar2010}}]\label{thm:Minty}
    Let $X$ be a~real Hilbert space, and let $A: X \rightrightarrows X$ be monotone.  
    Then $A$ is maximal monotone if, and only if, for some (and hence every) $\lambda>0$ we have that $A+\lambda I$ is surjective.
\end{thm}

Next, we state a result providing a criterion for the maximal monotonicity of the sum of two maximal monotone operators. See~\cite{Rockafellar1970} for the original reference. 

\begin{thm}[Rockafellar's theorem~{\cite[Thm.~2.6]{Bar2010}}]\label{sumthm}
    Let $X$ be a~real Hilbert space, and let $A,B:X\rightrightarrows X$ be maximal monotone with
    \[\big(\operatorname{int} \operatorname{dom}( A)\big)\cap \operatorname{dom}(B)\neq \emptyset.\] 
    Then $A+B:X\rightrightarrows X$ is maximal monotone.
\end{thm}

We recall that maximal monotonicity imposes a simple geometric structure on
the images of individual elements under the operator.

\begin{lem}[{Images under maximal monotone operators \cite[Prop.~20.36]{BauschkeCombettes2017}}]\label{lem:mm-values-closed-convex}
Let $X$ be a real Hilbert space and let $A:X\rightrightarrows X$ be maximal
monotone. Then, for every $x\in\dom(A)$, the image $Ax$ of~$x$ under~$A$ is a nonempty closed
convex subset of $X$.
\end{lem}

\subsection{Non-autonomous evolution equations}

Next we consider the non-autonomous abstract Cauchy problem
\begin{equation}\label{ACP-t}
    \begin{aligned}
        \dot{x}(t) &\in - A_t(x(t)), \quad t\in[0,T],\\
        x(0) &= x_0.
    \end{aligned}
\end{equation}
Throughout, we assume that $A_t:X\rightrightarrows X$ is maximal monotone for each $t\in[0,T]$. We also fix an initial value $x_0\in X$, which will be needed only later for the existence results and not for the definition of the solution concepts themselves.

\begin{defn}[Strong and generalized solutions, Evans~\cite{Evans1977}]\label{def:sol}
Let $X$ be a real Hilbert space, $T>0$, 
and let $A_t:X\rightrightarrows X$, $t\in[0,T]$.
\begin{enumerate}[label=(\alph*),ref=(\alph*)]
    \item (\emph{Strong solution})  
    A function $x\in W^{1,1}([0,T],X)$ is called a strong solution of~\eqref{ACP-t}, if 
    \[
        x(0)=x_0\in\overline{\dom(A_0)} \quad\text{and}\quad
        \dot{x}(t)\in -A_t(x(t))\ \text{for a.e. } t\in[0,T].
    \]
    \item (\emph{Generalized solution})  
    A function $x\in C([0,T],X)$ is called a generalized solution of~\eqref{ACP-t}, if there exists a sequence of partitions
    \[
      \mathcal{P}^n:\ 0=t_0^n<t_1^n<\dots<t_{m(n)}^n=T,\qquad \max_k\, (t_k^n-t_{k-1}^n)\xrightarrow[n\to\infty]{}0,
    \]
    and, for all $n\in\N$, a~sequence $(x_k^n)_{k=0}^{m(n)}\subset X$ which satisfies
    \begin{equation}
      x_0^n=x_0\quad \text{and}\quad \frac{x_k^n-x_{k-1}^n}{t_k^n-t_{k-1}^n} \in -A_{t_k^n}(x_k^n)\quad \text{for all}\ k=1,\ldots,m(n),
    \label{eq:discsscheme}
    \end{equation}
    such that the associated step functions
    \[
      x^n(t)=
      \begin{cases}
        x_0^n,& t=0,\\[2mm]
        x_k^n,& t\in (t_{k-1}^n,t_k^n],\ k=1,\ldots,m(n)
      \end{cases}
    \]
    satisfy that $(x^n)$ converges uniformly to $x$ in $C([0,T],X)$ as $n\to\infty$. 
\end{enumerate}
\end{defn}

In the case where $A_t:X\rightrightarrows X$ is maximal monotone, Minty's theorem yields that for all $\lambda>0$, $A_t+\lambda I:X\rightrightarrows X$ is surjective, i.e.,
\[
    \dom (A_t+\lambda I)^{-1} = \mathcal{R}(A_t+\lambda I) = X.
\]
Since monotonicity further yields that $A_t+\lambda I$ is injective (by a straightforward argument), we obtain that the inverse $(A_t+\lambda I)^{-1}$ is single-valued: Let $x\in X$ and $y,z\in (A_t+\lambda I)^{-1} x$, which means that $(y,x), (z,x)\in (A_t+\lambda I)$, i.e., $x\in (A_t+\lambda I)y \cap (A_t+\lambda I)z$, which by injectivity implies $y=z$. Therefore, $(A_t+\lambda I)^{-1}:X\to X$ is a single-valued operator. Note that, for every $\lambda>0$ and $x\in X$,
\[
(I+\lambda A_t)^{-1}x
=
(\lambda^{-1}I+A_t)^{-1}(\lambda^{-1}x),
\]
and hence the resolvent $(I+\lambda A_t)^{-1}:X\to X$ is well-defined and single-valued. Thus, the {\em Yosida approximation}
\[
A_{t,\lambda}:=\lambda^{-1}\big(I-(I+\lambda A_t)^{-1}\big):X\to X
\]
is a well-defined single-valued operator. We use this concept in the following, where we formulate assumptions for the existence of solutions to Cauchy problems of type~\eqref{ACP-t}.

\begin{assumption}[Regularity conditions for $t\mapsto A_t$]\label{ass:barbu-evans}
We consider two optional assumptions on the variation of the Yosida approximation $A_{t,\lambda}$ of $A_t$, $t\in[0,T]$, $\lambda>0$.
\begin{enumerate}[label=(\alph*),ref=(\alph*)]
    \item\label{ass:barbu-evansa}
    There exists $C>0$ such that
    \begin{equation}\label{eq:Atlamstrong}
      \|A_{t,\lambda} z-A_{s,\lambda} z\|
      \le C\, |t-s|\,\big(\|A_{t,\lambda} z\|+\|z\|+1\big) \quad\forall\, z\in X, \forall\,\lambda>0,\text{ and almost all }s,t\in[0,T].
    \end{equation}
    \item\label{ass:barbu-evansb} 
    There exist $\lambda_0>0$, $h\in L^1([0,T],X)$, and non-decreasing and continuous $L:\R_{\ge0}\to \R_{\ge0}$ such that
    \begin{equation}
          \|A_{t,\lambda} z-A_{s,\lambda} z\|
      \le  \|h(t)-h(s)\|\, L(\|z\|) \quad
      \forall\, z\in X, \forall\,\lambda\in(0,\lambda_0],\text{ and almost all }s,t\in[0,T].\label{eq:Atlamgen}
    \end{equation}
\end{enumerate}
\end{assumption}

\begin{thm}[Existence of strong and generalized solutions]\label{thm:barbu-evans}
Let $X$ be a real Hilbert space, $T>0$, and $A_t:X\rightrightarrows X$ be maximal monotone for all $t\in[0,T]$.
\begin{enumerate}[label=(\alph*),ref=(\alph*)]
    \item\label{thm:barbu-evansa} (\emph{Strong solutions, {\sc Barbu}~\cite[Thm.~4.19]{Bar2010}})  
    Under Assumption~\ref{ass:barbu-evans}\,\ref{ass:barbu-evansa}, 
    there exists a unique strong solution 
    $x\in W^{1,1}([0,T],X)$ of~\eqref{ACP-t} for every $x_0\in {\dom(A_0)}$.  
    \item\label{thm:barbu-evansb} (\emph{Generalized solutions, {\sc Evans}~\cite[Thms.~1 and~3, Rem.~8.4]{Evans1977}}) 
    Under Assumption~\ref{ass:barbu-evans}\,\ref{ass:barbu-evansb}, 
    there exists a unique generalized solution 
    $x\in C([0,T],X)$ of~\eqref{ACP-t} for every $x_0\in \overline{\dom(A_0)}$.
    \item  
    If both sets of assumptions hold, then the generalized and strong solutions coincide, see~\cite[Cor.~10.1]{Evans1977}.
\end{enumerate}
\end{thm}

\begin{lem}[Vanishing perturbations of the discretization]
\label{lem:vanishing-discrete-perturbations}
Let $X$ be a real Hilbert space and let
$A_t:X\rightrightarrows X$ be maximal monotone for every $t\in[0,T]$.
Suppose that there exist partitions
\[
0=t_0^n<t_1^n<\dots<t_{m(n)}^n=T,
\qquad
\max_k(t_k^n-t_{k-1}^n)\to0,
\]
sequences $(x_k^n)_{k=0}^{m(n)}\subset X$ with $x_0^n=x_0$, and
coefficients $(r_k^n)_{k=1}^{m(n)}\subset X$ such that
\[
\frac{x_k^n-x_{k-1}^n}{t_k^n-t_{k-1}^n}
\in
-A_{t_k^n}(x_k^n)+r_k^n,
\qquad k=1,\ldots,m(n),
\]
and the associated step functions satisfy
\[
x^n\to x\quad\text{in }C([0,T],X),
\qquad
r^n\to0\quad\text{in }L^1([0,T],X).
\]
Then $x$ is a generalized solution of~\eqref{ACP-t} in the sense of
Definition~\ref{def:sol}.
\end{lem}

\begin{proof}
Set $h_k^n:=t_k^n-t_{k-1}^n$ and
\[
J_k^n:=(I+h_k^nA_{t_k^n})^{-1}.
\]
Since $A_{t_k^n}$ is maximal monotone, $J_k^n$ is everywhere defined,
single-valued and non-expansive. Indeed, if
$x=J_k^np$ and $y=J_k^nq$, then there exist
$a\in A_{t_k^n}x$ and $b\in A_{t_k^n}y$ such that
\[
p=x+h_k^na,\qquad q=y+h_k^nb.
\]
Hence monotonicity gives
\[
\langle x-y,p-q\rangle
=
\|x-y\|^2+h_k^n\langle x-y,a-b\rangle
\ge \|x-y\|^2,
\]
and therefore
\[
\|J_k^np-J_k^nq\|\le \|p-q\|.
\]

The perturbed discretization can be written as
\[
x_k^n
=
J_k^n\bigl(x_{k-1}^n+h_k^nr_k^n\bigr).
\]
Define $\widehat x_0^n=x_0$ and recursively
\[
\widehat x_k^n=J_k^n\widehat x_{k-1}^n.
\]
Then $(\widehat x_k^n)$ satisfies the exact discretization
scheme~\eqref{eq:discsscheme}. By non-expansiveness,
\[
\|x_k^n-\widehat x_k^n\|
\le
\|x_{k-1}^n-\widehat x_{k-1}^n\|
+h_k^n\|r_k^n\|.
\]
Induction yields
\[
\|x_k^n-\widehat x_k^n\|
\le
\sum_{j=1}^kh_j^n\|r_j^n\|
\le
\|r^n\|_{L^1(0,T;X)}.
\]
Thus the associated step functions satisfy
\[
\|x^n-\widehat x^n\|_{C([0,T],X)}
\le
\|r^n\|_{L^1(0,T;X)}
\to0.
\]
Since $x^n\to x$ uniformly, also $\widehat x^n\to x$ uniformly.
The functions $\widehat x^n$ satisfy the exact scheme
\eqref{eq:discsscheme}, and therefore $x$ is a generalized solution
in the sense of Definition~\ref{def:sol}.
\end{proof}

\begin{rem}[Inhomogeneous terms]\label{rem:inhomogeneities}

We shall also use the corresponding notion of a generalized solution for
the inhomogeneous problem
\begin{equation}\label{eq:ACP-inhom}
\dot{x}(t)\in-A_t(x(t))+f(t),
\qquad x(0)=x_0,
\end{equation}
where $f\in L^1([0,T],X)$. A function $x\in W^{1,1}([0,T],X)$ is called a strong solution
of~\eqref{eq:ACP-inhom} if
\[
x(0)=x_0
\]
and
\[
\dot x(t)\in-A_t(x(t))+f(t)
\]
for almost every $t\in[0,T]$.
A function $x\in C([0,T],X)$ is called a generalized solution
of~\eqref{eq:ACP-inhom} if there exist partitions
\[
0=t_0^n<t_1^n<\dots<t_{m(n)}^n=T,
\qquad
\max_k(t_k^n-t_{k-1}^n)\to0,
\]
sequences $(x_k^n)_{k=0}^{m(n)}\subset X$ with $x_0^n=x_0$, and
coefficients $(f_k^n)_{k=1}^{m(n)}\subset X$ such that
\[
\frac{x_k^n-x_{k-1}^n}{t_k^n-t_{k-1}^n}
\in
-A_{t_k^n}(x_k^n)+f_k^n,
\qquad k=1,\ldots,m(n),
\]
and the associated step functions satisfy
\[
x^n\to x\quad\text{in }C([0,T],X),
\qquad
f^n\to f\quad\text{in }L^1([0,T],X).
\]
For $f=0$, this notion coincides with Definition~\ref{def:sol};
the nontrivial implication follows from
Lemma~\ref{lem:vanishing-discrete-perturbations}.
Likewise, the coefficients $f_k^n$ may be replaced by
$f_k^n+r_k^n$ whenever the corresponding step functions satisfy
$r^n\to0$ in $L^1([0,T],X)$.
This is the inhomogeneous discretization used in the existence theory
of~\cite{Evans1977}.

\end{rem}

\begin{cor}[Inhomogeneous generalized solutions]
\label{cor:inhomogeneous-generalized}
Let $X$ be a real Hilbert space, $T>0$, and let
$A_t:X\rightrightarrows X$ be maximal monotone for all $t\in[0,T]$.
Under Assumption~\ref{ass:barbu-evans}\,\ref{ass:barbu-evansb},
for every $x_0\in\overline{\dom(A_0)}$ and every
$f\in L^1([0,T],X)$ there exists a unique generalized solution
of~\eqref{eq:ACP-inhom} in the sense of
Remark~\ref{rem:inhomogeneities}.
\end{cor}
\begin{proof}
Existence follows from \cite[Thm.~1]{Evans1977}, while uniqueness
follows from the comparison estimate in
\cite[Thm.~3 and Rem.~8.4]{Evans1977}.
\end{proof}
Next we consider an operator function which appears in the context of the drill-string model.

\begin{lem}\label{lem:Atlam-translation-unbdd}
Let $X$ be a real Hilbert space. Let $D:\dom(D)\subset X\to X$ be linear and maximal monotone (e.g., skew-adjoint) and
$B:X\rightrightarrows X$ be maximal monotone with $0\in\operatorname{int}\dom(B)$.
Let $f:[0,T]\to X$ and define, for $t\in[0,T]$,
\[
A_t:\dom(D) \cap \dom(B)\rightrightarrows X,\qquad
A_t z := D z + B(z)+f(t).
\]
Then $A_t$ is maximal monotone for all $t\in[0,T]$. Moreover, the following holds for the Yosida approximation
$A_{t,\lambda}$ of $A_t$, for
all $z\in X$, all $\lambda>0$, and almost all $s,t\in[0,T]$:
\begin{equation}\label{eq:Atlam-translation-Lip}
\|A_{t,\lambda}z-A_{s,\lambda}z\|
\le \|f(t)-f(s)\|.
\end{equation}
In particular:
\begin{enumerate}[label=(\alph*),ref=(\alph*)]
\item\label{lem:Atlam-translation-unbdda} If $f\in W^{1,\infty}([0,T],X)$, then $t\mapsto A_t$ fulfills
Assumption~\ref{ass:barbu-evans}\,\ref{ass:barbu-evansa} with
\[
C:=\|\dot f\|_{L^\infty(0,T;X)}.
\]
\item\label{lem:Atlam-translation-unbddb} If $f\in L^1([0,T],X)$, then $t\mapsto A_t$ fulfills
Assumption~\ref{ass:barbu-evans}\,\ref{ass:barbu-evansb} with $\lambda_0=1$,
$L(\cdot)\equiv 1$, and $h=f$.
\end{enumerate}
\end{lem}
\begin{proof} Maximal monotonicity of $A_t$ is a direct consequence of Rockafellar's theorem (see Theorem~\ref{sumthm}). 
To prove~\eqref{eq:Atlam-translation-Lip}, let $z\in X$, $\lambda>0$, and set
$x_t:=J_{t,\lambda}z:=(I+\lambda A_t)^{-1}z\in\mathcal{R}((I+\lambda A_t)^{-1}) = \dom(A_t)=\dom(D)\cap\dom(B)$, $t\in[0,T]$.
Then there exist $b_t\in B(x_t)$ such that
\[
z-x_t=\lambda\Bigl(Dx_t+b_t+f(t)\Bigr)
\qquad\text{for all }t\in[0,T].
\]
Let $s,t\in[0,T]$. Then the latter relation gives
\[
x_t-x_s
=\lambda\Bigl(D(x_s-x_t)+(b_s-b_t)
+f(s)-f(t)\Bigr).
\]
Taking the inner product with $x_t-x_s$ and using that, by monotonicity of $D$, $\langle x_t-x_s,D(x_t-x_s)\rangle \ge 0$, we obtain
\[
\|x_t-x_s\|^2
\le \lambda\langle b_s-b_t,x_t-x_s\rangle +\lambda\langle f(s)-f(t),\,x_t-x_s\rangle.\]
Monotonicity of $B$ gives
$\langle b_s-b_t,\,x_t-x_s\rangle\le0$, and thus
\[
\|x_t-x_s\|^2
\le
\lambda\langle f(s)-f(t),\,x_t-x_s\rangle.
\]
The Cauchy--Schwarz inequality now yields $\|x_t-x_s\|
\le \lambda \|f(t)-f(s)\|$.
Recalling that $A_{t,\lambda}=\frac1\lambda(I-J_{t,\lambda})$, we conclude that
\[
\|A_{t,\lambda}z-A_{s,\lambda}z\|
=\frac1\lambda\|x_t-x_s\|
\le \|f(t)-f(s)\|.
\]
This proves \eqref{eq:Atlam-translation-Lip}. 
The statement \emph{(a)} follows immediately from $\|f(t)-f(s)\|\le \|\dot f\|_{L^\infty}|t-s|$, if $f\in W^{1,\infty}([0,T],X)$, whereas \emph{(b)} can be directly concluded from $f\in L^1([0,T],X)$.
\end{proof}
\subsection{Nonlinear abstract boundary control systems}\label{abstract boundary control}

We now consider a class of abstract differential equations that provides a theoretical framework for the drill string model~\eqref{drillstring}. To this end, let $X$ be a real Hilbert space, let $\mathcal D\subset X$ be a linear subspace, let $\mathfrak{A}:X\rightrightarrows X$ satisfy $\dom(\mathfrak A)\subset\mathcal D$, let $\mathfrak{p}:\mathcal D\to\R^m$ be linear and surjective, and let $u\in W^{1,1}([0,T],\R^m)$. Then we study
\begin{equation}\label{eq:bndcont}
\begin{aligned}
       \dot{x}(t)&\in -\mathfrak{A}x(t), &
       \mathfrak{p}x(t)&=u(t),&& t\in[0,T],\\
       x(0)&=x_0\in X.
\end{aligned} 
\end{equation}
Note that this system can be written as a 
non-autonomous abstract Cauchy problem~\eqref{ACP-t} 
with time-dependent operators 
$A_t:X\rightrightarrows X$ given by
$A_t=\setdef{(x,z)\in \mathfrak{A}}{\mathfrak{p}x=u(t)}$.
The solution concepts introduced in Definition~\ref{def:sol} carry over accordingly.  


Next, we show that the boundary data can be eliminated by a suitable affine transformation, leading to an equivalent evolution problem with constant domain and an induced inhomogeneity.
The underlying idea is classical and goes back to {\sc Fattorini}, who applied this approach to linear systems in~\cite{Fattorini1968}. Here, we present a corresponding formulation tailored to our nonlinear and set-valued setting.

\begin{lem}[Equivalence via lifting]\label{lem:equivalence-lifting}
Let $T>0$ and $u\in W^{1,1}([0,T],\R^m)$ be given, and let
$X$, $\mathcal D$, $\mathfrak A$, and $\mathfrak p$ be as above.
Let $p_1,\ldots,p_m\in\mathcal D$ satisfy
\[
\mathfrak p p_i=e_i,\qquad i=1,\ldots,m,
\]
where $e_i\in\R^m$ denotes the $i$th canonical unit vector.
Assume furthermore that
\begin{equation}\label{eq:domain-shift-invariance}
x\in\dom(\mathfrak A)
\quad\Longleftrightarrow\quad
x+\sum_{i=1}^m a_ip_i\in\dom(\mathfrak A)
\qquad
\text{for all }a_1,\ldots,a_m\in\R.
\end{equation}
Define
\begin{equation}\label{eq:tildexbnd}
\widetilde x(t):=\sum_{i=1}^mp_iu_i(t),
\qquad t\in[0,T].
\end{equation}
Then $\widetilde x\in W^{1,1}([0,T],X)$ and
$\mathfrak p\widetilde x(t)=u(t)$.

For $t\in[0,T]$, define
\[
\mathcal B_t z
:=
\mathfrak A\bigl(z+\widetilde x(t)\bigr),
\qquad
\dom(\mathcal B_t)
=
\setdef{z\in\dom(\mathfrak A)}{\mathfrak pz=0}.
\]
Then $\dom(\mathcal B_t)$ is independent of $t$.

A function $x$ is a strong solution of~\eqref{eq:bndcont} if and only if
$z:=x-\widetilde x$ is a strong solution of
\begin{equation}\label{eq:lifted-inhomogeneous}
\dot z(t)\in-\mathcal B_tz(t)-\dot{\widetilde x}(t),
\qquad
z(0)=x_0-\widetilde x(0).
\end{equation}
If, in addition, $\mathcal B_t$ is maximal monotone for every
$t\in[0,T]$, then the same equivalence holds for generalized solutions,
where~\eqref{eq:lifted-inhomogeneous} is understood in the sense of
Remark~\ref{rem:inhomogeneities}.
\end{lem}

\begin{proof}
By construction,
\[
\mathfrak p\widetilde x(t)
=
\sum_{i=1}^mu_i(t)\mathfrak pp_i
=
u(t).
\]
Assumption~\eqref{eq:domain-shift-invariance} shows that translation by
$\widetilde x(t)$ is a bijection from
\[
\setdef{z\in\dom(\mathfrak A)}{\mathfrak pz=0}
\]
onto
\[
\setdef{x\in\dom(\mathfrak A)}{\mathfrak px=u(t)}.
\]
This proves, in particular, that the domain of $\mathcal B_t$ is
independent of $t$.

For strong solutions, let first $x$ solve~\eqref{eq:bndcont} and set
$z=x-\widetilde x$. Then
\[
\dot z(t)
=
\dot x(t)-\dot{\widetilde x}(t)
\in
-\mathfrak A(x(t))-\dot{\widetilde x}(t)
=
-\mathcal B_tz(t)-\dot{\widetilde x}(t)
\]
for almost every $t$. Conversely, if $z$ solves
\eqref{eq:lifted-inhomogeneous}, then
$x=z+\widetilde x$ satisfies
\[
\dot x(t)\in-\mathfrak A(x(t))
\]
for almost every $t$, and
\[
\mathfrak px(t)=\mathfrak pz(t)+u(t)=u(t).
\]

It remains to prove the assertion for generalized solutions.
Let first $x$ be a generalized solution of~\eqref{eq:bndcont}.
Then there exist partitions
\[
0=t_0^n<t_1^n<\dots<t_{m(n)}^n=T
\]
with mesh size tending to zero and sequences $(x_k^n)$ such that
$x^n\to x$ uniformly and
\[
\frac{x_k^n-x_{k-1}^n}{h_k^n}
\in
-\mathfrak A(x_k^n),
\qquad
\mathfrak px_k^n=u(t_k^n),
\]
where $h_k^n:=t_k^n-t_{k-1}^n$.
Set
\[
z_k^n:=x_k^n-\widetilde x(t_k^n),
\qquad
d_k^n:=
\frac{\widetilde x(t_k^n)-\widetilde x(t_{k-1}^n)}{h_k^n}.
\]
Then, for $k=1,\ldots,m(n)$,
$z_k^n\in\dom(\mathcal B_{t_k^n})$, and
\[
\frac{z_k^n-z_{k-1}^n}{h_k^n}
\in
-\mathcal B_{t_k^n}(z_k^n)-d_k^n.
\]
Since $\widetilde x\in W^{1,1}([0,T],X)$,
\[
d_k^n
=
\frac1{h_k^n}
\int_{t_{k-1}^n}^{t_k^n}\dot{\widetilde x}(s)\,ds.
\]
Thus the associated step functions $d^n$ are the cell-average
approximations of $\dot{\widetilde x}$ and satisfy
\[
d^n\to\dot{\widetilde x}
\qquad\text{in }L^1([0,T],X).
\]
Indeed, the cell-averaging operators are contractions on $L^1$;
the assertion follows first for continuous functions by uniform
continuity and then for arbitrary $L^1$ functions by density.

Let $\widetilde x^n$ denote the right-endpoint step approximation of
$\widetilde x$, i.e.,
\[
\widetilde x^n(t)=\widetilde x(t_k^n)
\quad\text{for }t\in(t_{k-1}^n,t_k^n].
\]
Then
\[
z^n=x^n-\widetilde x^n.
\]
Since $\widetilde x$ is uniformly continuous on $[0,T]$,
$\widetilde x^n\to\widetilde x$ uniformly, and hence
\[
z^n\to x-\widetilde x=z
\qquad\text{in }C([0,T],X).
\]
Therefore $z$ is a generalized solution of
\eqref{eq:lifted-inhomogeneous} in the sense of
Remark~\ref{rem:inhomogeneities}.

Conversely, let $z$ be a generalized solution of
\eqref{eq:lifted-inhomogeneous}. Then there exist partitions,
sequences $(z_k^n)$ and coefficients $(f_k^n)$ such that
\[
z^n\to z\quad\text{uniformly},
\qquad
f^n\to-\dot{\widetilde x}\quad\text{in }L^1([0,T],X),
\]
and
\[
\frac{z_k^n-z_{k-1}^n}{h_k^n}
\in
-\mathcal B_{t_k^n}(z_k^n)+f_k^n.
\]
Define
\[
x_k^n:=z_k^n+\widetilde x(t_k^n).
\]
Then
\[
\mathfrak px_k^n=u(t_k^n)
\]
and
\[
\frac{x_k^n-x_{k-1}^n}{h_k^n}
\in
-\mathfrak A(x_k^n)+f_k^n+d_k^n.
\]
Since
\[
f^n+d^n\to0
\qquad\text{in }L^1([0,T],X),
\]
this is a discretization of the boundary system with a vanishing
$L^1$ perturbation.

Let $\widetilde x^n$ denote the right-endpoint step approximation
of $\widetilde x$ associated with the present partition. Then
\[
x^n=z^n+\widetilde x^n.
\]
Since $\widetilde x^n\to\widetilde x$ uniformly on $[0,T]$ and
$z^n\to z$ uniformly, we obtain
\[
x^n\to z+\widetilde x=x
\qquad\text{uniformly on }[0,T].
\]
For fixed $t$, the fiber operator
\[
A_t^u
:=
\setdef{(x,y)\in\mathfrak A}{\mathfrak px=u(t)}
\]
is obtained from $\mathcal B_t$ by a translation of the first
component by $\widetilde x(t)$. Hence maximal monotonicity of
$\mathcal B_t$ implies maximal monotonicity of $A_t^u$.
Lemma~\ref{lem:vanishing-discrete-perturbations} therefore shows that
$x=z+\widetilde x$ is a generalized solution of~\eqref{eq:bndcont}.
\end{proof}

Based on the previous lemma and in combination with Theorem~\ref{thm:barbu-evans}, existence results for the abstract boundary control system~\eqref{eq:bndcont} can be established.
We now consider a class of boundary control systems.
While this class may appear rather specialized, it captures exactly the structure needed for the drill-string model.

\begin{prop}\label{prop:specbcs}
Consider the abstract boundary control system~\eqref{eq:bndcont}, where $X$ is a real Hilbert space and
\[
\mathfrak{A}:X\rightrightarrows X\quad \text{with}\quad \mathfrak{A}x = D_f x + B(x),
\]
where $D_f:\dom(D_f)\subset X\to X$ is a linear operator, 
$B:X\rightrightarrows X$ is maximal monotone with $0\in\operatorname{int}\dom(B)$, and $\mathfrak{p}:\dom(D_f)\to\R^m$ is linear and surjective.
Assume that the following properties hold:
\begin{enumerate}[label=(\roman*),ref=(\roman*)]
\item The restriction $D$ of $D_f$ to $\ker \mathfrak{p}$ is maximal monotone (this holds in
particular if $D$ is skew-adjoint).
\item There exist $p_1,\ldots,p_m\in\dom(D_f)$ satisfying
$\mathfrak{p}p_i=e_i$, $i=1,\ldots,m$,
where $e_i\in\R^m$ denotes the $i$th canonical unit vector, such that
\begin{equation}
\forall\, x\in \dom(B) \ \forall\, u_1,\ldots,u_m\in\R:\ x+\sum_{i=1}^m u_ip_i\in \dom(B)\ \wedge\ B(x)=B\left(x+\sum_{i=1}^m u_ip_i\right).
\label{eq:B2shift}
\end{equation}
\end{enumerate}
Then the following assertions hold:
\begin{enumerate}[label=(\alph*),ref=(\alph*)]
\item If $u\in W^{2,\infty}([0,T],\R^m)$ and $x_0\in \dom(D_f)\cap\dom(B)$ with $\mathfrak{p}x_0=u(0)$, then
\eqref{eq:bndcont} has a~unique strong solution.
\item If $u\in W^{1,1}([0,T],\R^m)$ and
\[
x_0\in
\overline{
\setdef{x\in\dom(D_f)\cap\dom(B)}
{\mathfrak p x=u(0)}
},
\]
then \eqref{eq:bndcont} has a~unique generalized solution.
\end{enumerate}
\end{prop}

\begin{proof}
Let $\widetilde{x}$ be defined as in \eqref{eq:tildexbnd}.
Since $\dom(D_f)$ is a linear space and~\eqref{eq:B2shift} is
invariant under arbitrary real shifts in the directions
$p_1,\ldots,p_m$, condition~\eqref{eq:domain-shift-invariance} of
Lemma~\ref{lem:equivalence-lifting} is satisfied with
$\mathcal D=\dom(D_f)$.
Consider the operator function $t\mapsto A_t$, $t\in[0,T]$, with $\dom(A_t)=\dom(D)\cap\dom(B)$ and
\[
A_t z=
Dz+B\bigl(z+\tilde x(t)\bigr)
+D_f\widetilde{x}(t)
+\dot{\widetilde{x}}(t),
\quad z\in\dom(A_t).
\]
If $u\in W^{2,\infty}([0,T],\R^m)$ and 
$x_0\in \dom(D_f)\cap\dom(B)$ with $\mathfrak{p}x_0=u(0)$, then 
$z_0:=x_0-\widetilde x(0)\in\ker \mathfrak{p}$ and since $x_0\in\dom (D_f)$ and $\widetilde x(0)\in\dom(D_f)$ (by $p_i\in\dom(D_f)$) we have $z_0 \in \dom(D)$. Moreover, $z_0 \in \dom(B)$ by $x_0\in\dom(B)$ and~\eqref{eq:B2shift}, and thus $z_0\in \dom(A_0)$. Further, by invoking~\eqref{eq:B2shift}, we have, for all $z\in X$, $t\in[0,T]$, 
\[B\bigl(z+\tilde x(t)\bigr)=B(z).\]
Hence, by Lemma~\ref{lem:Atlam-translation-unbdd}\,\ref{lem:Atlam-translation-unbdda}, the operator function $t\mapsto A_t$ fulfills
Assumption~\ref{ass:barbu-evans}\,\ref{ass:barbu-evansa}. Then Theorem~\ref{thm:barbu-evans}\,\ref{thm:barbu-evansa} yields that the initial value problem $\dot{z}(t) \in - A_t(z(t))$ with $z(0)=z_0$ has a~unique strong solution $z\in W^{1,1}([0,T],X)$. Since the strong solutions of $\dot{z}(t) \in - A_t(z(t))$ are, by Lemma~\ref{lem:equivalence-lifting}, related 
to those of
$\dot{x}(t)\in -\mathfrak{A}x(t)$, 
$\mathfrak{p}x(t)=u(t)$ by $x=z+\tilde{x}$, we obtain that the boundary control system \eqref{eq:bndcont} has a~unique strong solution.\\

It remains to prove the assertion for generalized solutions.
Assume that $u\in W^{1,1}([0,T],\R^m)$ and set
\[
\widehat A_tz
:=
Dz+B(z)+D_f\widetilde x(t),
\qquad
\dom(\widehat A_t)=\dom(D)\cap\dom(B).
\]
Set $z_0:=x_0-\widetilde x(0)$. By the assumption on $x_0$ and
the shift property~\eqref{eq:B2shift}, translation by
$\widetilde x(0)$ maps
\[
\dom(D)\cap\dom(B)
\]
bijectively onto
\[
\setdef{x\in\dom(D_f)\cap\dom(B)}{\mathfrak p x=u(0)}.
\]
Consequently,
\[
z_0\in\overline{\dom(D)\cap\dom(B)}.
\]
Since $D_f\widetilde x\in L^1([0,T],X)$,
Lemma~\ref{lem:Atlam-translation-unbdd}\,
\ref{lem:Atlam-translation-unbddb} yields that
$t\mapsto\widehat A_t$ satisfies
Assumption~\ref{ass:barbu-evans}\,
\ref{ass:barbu-evansb}.
Moreover,
$\dot{\widetilde x}\in L^1([0,T],X)$.
Hence Corollary~\ref{cor:inhomogeneous-generalized} yields a unique
generalized solution of
\[
\dot z(t)
\in
-\widehat A_tz(t)-\dot{\widetilde x}(t).
\]
By~\eqref{eq:B2shift}, this equation is precisely the lifted system
from Lemma~\ref{lem:equivalence-lifting}.
Furthermore, $\widehat A_t$ is maximal monotone for every $t$, again
by Lemma~\ref{lem:Atlam-translation-unbdd}.
The generalized-solution part of
Lemma~\ref{lem:equivalence-lifting} therefore yields a unique
generalized solution of the boundary control system~\eqref{eq:bndcont}.
\end{proof}

\section{The open-loop system}\label{Sec:wellpose}
We now investigate the existence of solutions to the drill string model \eqref{drillstring} (in open-loop configuration, i.e., without a controller) by means of the previously outlined theory of nonlinear evolution equations. In doing so, we analyze the slightly more general case where the damping is described by a relation; which is evident from the first two equations in \eqref{drillstring}.

\subsection{Physical assumptions}
We begin by formulating the assumptions on the physical parameters.

\begin{assumption}\label{assump}\
\begin{enumerate}[label=(\alph*),ref=(\alph*)]
    \item\label{assumpa}\
  The mass density and shear modulus satisfy
\begin{align*}
   & \rho(\cdot),\rho^{-1}(\cdot),G(\cdot), G^{-1}(\cdot) \in L^{\infty}([0,\ell],\R), \\
   & \rho(\xi), G(\xi)>0 \quad \text{for a.e. } \xi \in [0,\ell].
\end{align*}
\item\label{assumpb}\
 The damping torque $F_d$ is a pointwise multivalued real map that satisfies the following assumption:
\begin{enumerate}[label=(\roman*),ref=(\roman*)]
  \item\label{assumpb1}
 For all $\xi\in[0,\ell]$, $F_d(\xi,\cdot):\R\rightrightarrows\R$ with $\operatorname{dom}(F_d(\xi,\cdot))=\R$.
  \item\label{assumpb2}
 For a.e.\ $\xi\in[0,\ell]$, $F_d(\xi,\cdot)$ is maximal monotone.
  \item\label{assumpb3}
The multivalued map
\[
  [0,\ell]\rightrightarrows\R^2,\qquad
  \xi\mapsto \operatorname{graph}(F_d(\xi,\cdot))
  :=\setdef{(v,F)\in\R^2}{F\in F_d(\xi,v)}
\]
is measurable, that is, for every open set $O\subset\R^2$ the set
\[
  \setdef{\xi\in[0,\ell]}
  {\operatorname{graph}(F_d(\xi,\cdot))\cap O\neq\emptyset}
\]
is Lebesgue measurable.
\item\label{assumpb4}
 There exist $\alpha\ge 0$ and $\beta\in L^2([0,\ell],\R)$ such that
  for a.e.\ $\xi\in[0,\ell]$ and all $v\in\R$ there exists $F\in F_d(\xi,v)$ with
  \[
    |F|\le \alpha\,|v|+\beta(\xi).
  \]
\end{enumerate}
\item\label{assumpc}
 The rotary inertia and the second moment of area at $\xi=\ell$ are positive, i.e., $J>0$ and   
$\Gamma>0$.
\item\label{assumpd}
 The damping force at the cutting end defines a~maximal monotone relation on a~neighborhood of $0$. That is, $F_e:\R\rightrightarrows \R$ is maximal monotone with $0\in\operatorname{int}(\operatorname{dom}(F_e))$.
\end{enumerate}
\end{assumption}
In what follows, we focus on the distributed damping relation $F_d$. In particular, we show that the  multivalued map
\begin{equation}\label{eq:Nemytskii}
  \widetilde F_d := \setdef{(v,F)\in L^2([0,\ell],\R)\times L^2([0,\ell],\R)}{F(\xi)\in F_d(\xi,v(\xi))\ \text{for a.e.\ }\xi\in[0,\ell]}
\end{equation}
is maximal monotone and satisfies $\dom(\widetilde F_d)=L^2([0,\ell],\R)$. Set-valued mappings of this type on function spaces are commonly referred to as \emph{Nemytskii relations}; see, e.g., \cite{Tolstonogov2021}.

\begin{prop}[Distributed damping as an everywhere defined maximal monotone Nemytskii relation]\label{prop:distrdamp}
Let $F_d:[0,\ell]\times\R\rightrightarrows\R$ satisfy Assumption~\ref{assump}\,\ref{assumpb}. Then the multivalued map
$\widetilde F_d : L^2([0,\ell],\R)\rightrightarrows L^2([0,\ell],\R)$ as in \eqref{eq:Nemytskii} is maximal monotone with $\dom(\widetilde F_d)=L^2([0,\ell],\R)$.
\end{prop}
\begin{proof}
\emph{Step 1 (Everywhere defined).}
Let $v\in L^2([0,\ell],\R)$. Define $H:[0,\ell]\rightrightarrows \R$ by
\[
H(\xi):=F_d(\xi,v(\xi)),\qquad \xi\in[0,\ell].
\]
Then, by using maximal monotonicity of the scalar relations (see
Assumption~\ref{assump}\,\ref{assumpb}\,\ref{assumpb2}) together with the full pointwise domain
(Assumption~\ref{assump}\,\ref{assumpb}\,\ref{assumpb1}), Lemma~\ref{lem:mm-values-closed-convex} implies that, for a.e.\ $\xi\in[0,\ell]$, the set
$H(\xi)\subset\R$ is nonempty, closed and convex.
In other words, for a.e.\ $\xi\in[0,\ell]$, $H(\xi)\subset\R$ is a~closed and nonempty interval.
By the measurability assumption on the multivalued map $\xi\mapsto \operatorname{graph}(F_d(\xi,\cdot))$
and the measurability of $v$, it holds that $H:[0,\ell]\rightrightarrows\R$ 
is measurable as well. Hence, for every open set $I\subset\R$, the set
\[
  \setdef{\xi\in[0,\ell]}{H(\xi)\cap I\neq\emptyset}
\]
is Lebesgue measurable. Now define $a,b:[0,\ell]\to\R\cup\{-\infty,\infty\}$ by
\[
  a(\xi):=\inf H(\xi),\qquad b(\xi):=\sup H(\xi).
\]
Then $a$ and $b$ are measurable since, for every $r\in\R$,
\begin{align*}
  \setdef{\xi\in[0,\ell]}{a(\xi)<r}
  &=
  \setdef{\xi\in[0,\ell]}{H(\xi)\cap(-\infty,r)\neq\emptyset},
\\
  \setdef{\xi\in[0,\ell]}{b(\xi)>r}
  &=
  \setdef{\xi\in[0,\ell]}{H(\xi)\cap(r,\infty)\neq\emptyset}.
\end{align*}
Now choose the \emph{minimal-magnitude selector}
\[
F(\xi):=\min\bigl\{\max\{0,a(\xi)\},\,b(\xi)\bigr\},\qquad \text{for a.e.\ }\xi\in[0,\ell].
\]
Since each $H(\xi)$ is a non-empty and closed interval, $F(\xi)$ is the element of $H(\xi)$ closest to zero.
Because $a$ and $b$ are measurable, $F$ naturally inherits measurability.
Further, closedness of $H(\xi)$ yields that $F(\xi)\in H(\xi)=F_d(\xi,v(\xi))$ for a.e.\ $\xi\in [0,\ell]$.
By the linear $L^2$-growth condition (Assumption~\ref{assump}\,\ref{assumpb}\,\ref{assumpb4}) we have for a.e.\ $\xi\in[0,\ell]$ that there exists
 $\widehat F(\xi)\in H(\xi)$ with $|\widehat F(\xi)|\le \alpha|v(\xi)|+\beta(\xi)$. Then minimality gives
$|F(\xi)|\le |\widehat F(\xi)|\le \alpha|v(\xi)|+\beta(\xi)$, whence $F\in L^2([0,\ell],\R)$. In other words, $v\in \operatorname{dom}(\widetilde F_d)$.\\
\emph{Step 2 (Maximal monotonicity).}
It particularly follows from Step~1 that $\operatorname{dom}(\widetilde F_d)\neq\emptyset$. Now additionally invoking Assumption~\ref{assump}\,\ref{assumpb}\,\ref{assumpb1}--\ref{assumpb3}, maximal monotonicity of $\widetilde F_d$ can be concluded from \cite[Thm~3.1]{Tolstonogov2021}.
\end{proof}

\subsection{The drill-string model: Solution existence and uniqueness}

Here we provide a solvability analysis for the drill-string model~\eqref{drillstring}.
To this end, we first reformulate the model as a first-order evolution system.
Specifically, we choose as state variables
the torsional strain $\tau(\cdot,t)=\phi_\xi(\cdot,t)$,
the angular momentum density $L(\cdot,t)=\rho(\cdot) \phi_t(\cdot,t)$,
and the bit angular momentum $L_b(t)=J\phi_t(\ell,t)$.
With this choice, the drill-string model~\eqref{drillstring} can be written as 
\begin{equation}\label{absfor1}
\begin{aligned}
\frac{\partial}{{\partial}t}\begin{pmatrix}
    \tau(\xi,t)\\
    L(\xi,t)\\
    {L_b}(t)
\end{pmatrix}& \in\begin{pmatrix}
    \tfrac{\partial}{{\partial}\xi}\big(\rho(\xi)^{-1}L(\xi,t)\big)\\
    \tfrac{\partial}{{\partial}\xi}\big(G(\xi)\tau(\xi,t)\big)-F_d(\xi,\rho(\xi)^{-1}L(\xi,t))\\
    -\Gamma (G\tau)(\ell,t)-F_e(J^{-1}L_b(t))
\end{pmatrix},\\
(\rho^{-1}L)(\ell,t)&=J^{-1}L_b,\\
(G\tau)(0,t)&=u(t).\\
\end{aligned}
\end{equation}
We show that this defines
an abstract boundary control system \eqref{eq:bndcont} on the space
\begin{subequations}\label{eq:Xspace}
\begin{equation}
X = L^2([0,\ell],\R)^2 \times \R
\end{equation}
endowed with the norm 
\begin{equation}
\bigg\|\left(\begin{smallmatrix}
\tau \\ L \\ L_b
\end{smallmatrix}\right)\bigg\|_X=\left(\int_0^\ell(G(\xi)\tau(\xi)^2+\rho(\xi)^{-1}L(\xi)^2)d\xi+J^{-1}\Gamma^{-1}L_b^2\right)^{1/2}.
\end{equation}
\end{subequations}
This norm is induced by an inner product. By our assumptions on $G$, $\rho$, $\Gamma$ and $J$, the latter norm is equivalent to the standard one in 
$X=L^2([0,\ell],\R)^2 \times \R$. In particular, $X$ is a~Hilbert space. In physical terms, the square of this norm expresses the energy stored in the system.

The operators in the abstract boundary control formulation \eqref{eq:bndcont}
are given by $\mathfrak{A}:X\rightrightarrows X$ and
$\mathfrak{p}:\dom(D_f)\to\R$ with
\begin{subequations}
    \label{eq:drilloperators}
\begin{equation}
\begin{aligned}
    \mathfrak{A}x&=\underbrace{\begin{pmatrix}
    -\tfrac{\partial}{{\partial}\xi}\big(\rho(\cdot)^{-1}L(\cdot)\big)\\
    -\tfrac{\partial}{{\partial}\xi}\big(G(\cdot)\tau(\cdot)\big)\\
    \Gamma (G\tau)(\ell)
\end{pmatrix}}_{=:D_fx}+\underbrace{\begin{pmatrix}
   0\\
    F_d(\cdot,\rho(\cdot)^{-1}L(\cdot))\\
    F_e(J^{-1}L_b)
\end{pmatrix}}_{=:B(x)},\quad x=\begin{pmatrix}
    \tau(\cdot)\\
    L(\cdot)\\
    {L_b}
\end{pmatrix},\\
\mathfrak{p}x&=(G\tau)(0),
\end{aligned}
\end{equation}
where $D_f$ and $B$ are defined on 
\begin{equation}
\begin{aligned}
\dom(D_f)&=\setdef{\left(\begin{smallmatrix} \tau\\L\\L_b \end{smallmatrix}\right) \in X}{G\tau, {\rho}^{-1}L\in W^{1,2}([0,\ell],\R)\,\wedge\, ({\rho}^{-1}L)(\ell)={J}^{-1}L_b},\\
\dom(B)&=\setdef{\left(\begin{smallmatrix} \tau\\L\\L_b \end{smallmatrix}\right) \in X}{{J}^{-1}L_b\in \dom F_e},
\end{aligned}
\end{equation}
\end{subequations}
resp. Strong and generalized solutions of \eqref{absfor1} are defined as the
corresponding solution concepts for the abstract boundary control system
\eqref{eq:bndcont} with operators given in \eqref{eq:drilloperators}.
The following theorem addresses existence and uniqueness of solutions.
Since the abstract prerequisites have been established in the previous section,
it suffices for the solvability results to verify that the present system
satisfies the assumptions of Proposition~\ref{prop:specbcs}.

\begin{thm}\label{thm:open_loop_existence}
    Consider the drill-string model \eqref{absfor1} so that the parameters satisfy Assumptions~\ref{assump}. Let $u:[0,T]\to\R$ and $\tau_{0},L_0\in L^2([0,\ell],\R)$, $L_{b0}\in\R$. Then the following holds:
\begin{enumerate}[label=(\alph*),ref=(\alph*)]
\item If $u\in W^{2,\infty}([0,T],\R)$ and $G\tau_0,\rho^{-1}L_0\in W^{1,2}([0,\ell],\R)$ with $(\rho^{-1}L_0)(\ell)=J^{-1}L_{b0}\in\dom(F_e)$ and $(G\tau_0)(0)=u(0)$, then 
the drill-string model \eqref{absfor1} has a~unique strong solution with $\tau(0,\cdot)=\tau_0(\cdot)$, $L(0,\cdot)=L_0(\cdot)$ and $L_b(0)=L_{b0}$.
\item If $u\in W^{1,1}([0,T],\R)$  and $\tau_0,L_0\in L^{2}([0,\ell],\R)$ with $J^{-1}L_{b0}\in\dom(F_e)$,
then the drill-string model \eqref{absfor1} has a~unique generalized solution with $\tau(0,\cdot)=\tau_0(\cdot)$, $L(0,\cdot)=L_0(\cdot)$ and $L_b(0)=L_{b0}$.
\end{enumerate}
\end{thm}
\begin{proof}
Consider $D_f:\dom(D_f)\subset X\to X$, $B:X \rightrightarrows X$ and $\mathfrak{p}:\dom(D_f)\to\R$ as in \eqref{eq:drilloperators}. By Proposition~\ref{prop:specbcs} the assertion is shown, if we prove that:
\begin{enumerate}[label=(\roman*),ref=(\roman*)]
\item The operator $D:\dom(D)\subset X\to X$, $x\mapsto D_fx$ with
\begin{equation}
  \dom(D)=  \setdef{\left(\begin{smallmatrix} \tau\\L\\L_b \end{smallmatrix}\right) \in X}{G\tau, {\rho}^{-1}L\in W^{1,2}([0,\ell],\R)\,\wedge\, ({\rho}^{-1}L)(\ell)={J}^{-1}L_b \,\wedge\,(G\tau)(0)=0}
\end{equation}
is skew-adjoint.
\item $B:X \rightrightarrows X$ is maximal monotone with $0\in\operatorname{int}\dom(B)$.
\item There exists $p\in \dom(D_f)$ with $\mathfrak{p}(up)=u$, $x+up\in\dom(B)$ and $B(x)=B(x+up)$ for all $x\in \dom(B)$ and $u\in\R$. 
\end{enumerate}
We now verify these conditions one by one.
\begin{enumerate}[label=(\roman*),ref=(\roman*)]
\item 
{\em Step~1:} $D$ is skew-symmetric.
Let $x_1,x_2\in\dom(D)$, and denote
\begin{equation}
x_1=\left(\begin{smallmatrix}\tau_1\\L_1\\L_{b1}\end{smallmatrix}\right),
\qquad
x_2=\left(\begin{smallmatrix}\tau_2\\L_2\\L_{b2}\end{smallmatrix}\right).
\label{eq:x1x2}\end{equation}
Using the product rule and the definition of the inner product induced by \eqref{eq:Xspace}, we compute
\begin{align*}
&\langle x_1, Dx_2\rangle_X+\langle Dx_1, x_2\rangle_X\\
&=\int_0^\ell\Big(
-G\tau_1\,\tfrac{\partial}{\partial\xi}\big(\rho^{-1}L_2\big)
-\rho^{-1}L_1\,\tfrac{\partial}{\partial\xi}\big(G\tau_2\big)
\Big)\,d\xi
+J^{-1}L_{b1}\,(G\tau_2)(\ell)\\
&\quad+\int_0^\ell\Big(
-G\tau_2\,\tfrac{\partial}{\partial\xi}\big(\rho^{-1}L_1\big)
-\rho^{-1}L_2\,\tfrac{\partial}{\partial\xi}\big(G\tau_1\big)
\Big)\,d\xi
+J^{-1}L_{b2}\,(G\tau_1)(\ell)\\
&=-\int_0^\ell \tfrac{\partial}{\partial\xi}\big(G\tau_1\,\rho^{-1}L_2\big)\,d\xi
-\int_0^\ell \tfrac{\partial}{\partial\xi}\big(G\tau_2\,\rho^{-1}L_1\big)\,d\xi
+J^{-1}\Big(L_{b1}\,(G\tau_2)(\ell)+L_{b2}\,(G\tau_1)(\ell)\Big)\\
&=-\Big[ G\tau_1\,\rho^{-1}L_2+G\tau_2\,\rho^{-1}L_1\Big]_{0}^{\ell}
+J^{-1}\Big(L_{b1}\,(G\tau_2)(\ell)+L_{b2}\,(G\tau_1)(\ell)\Big).
\end{align*}
Since $x_1,x_2\in\dom(D)$, we have $(G\tau_1)(0)=(G\tau_2)(0)=0$ and
$(\rho^{-1}L_1)(\ell)=J^{-1}L_{b1}$, $(\rho^{-1}L_2)(\ell)=J^{-1}L_{b2}$.
Hence the contribution at $\xi=0$ vanishes, and at $\xi=\ell$ we obtain
\begin{align*}
&-\Big((G\tau_1)(\ell)\,(\rho^{-1}L_2)(\ell)+(G\tau_2)(\ell)\,(\rho^{-1}L_1)(\ell)\Big)
+J^{-1}\Big(L_{b1}\,(G\tau_2)(\ell)+L_{b2}\,(G\tau_1)(\ell)\Big)\\
&\qquad
=-J^{-1}\Big((G\tau_1)(\ell)\,L_{b2}+(G\tau_2)(\ell)\,L_{b1}\Big)
+J^{-1}\Big(L_{b1}\,(G\tau_2)(\ell)+L_{b2}\,(G\tau_1)(\ell)\Big)=0.
\end{align*}
Therefore
\[
\langle x_1, Dx_2\rangle_X=-\langle Dx_1, x_2\rangle_X
\qquad\text{for all }x_1,x_2\in\dom(D),
\]
and hence $D$ is skew-symmetric.
Moreover, $\dom(D)$ is dense in $X$. Indeed, for arbitrary
$\tau,L\in L^2([0,\ell],\R)$ and $L_b\in\R$, the functions
$G\tau$ and $\rho^{-1}L$ can be approximated in $L^2$ by functions
$q_n,r_n\in W^{1,2}([0,\ell],\R)$ satisfying
\[
q_n(0)=0,
\qquad
r_n(\ell)=J^{-1}L_b,
\]
respectively. Then
\[
\left(
G^{-1}q_n,\,
\rho r_n,\,
L_b
\right)^\top\in\dom(D)
\]
and converges to $(\tau,L,L_b)^\top$ in $X$.

{\em Step 2:} $D^*=-D$.
Our result from the first step implies that $\dom(D)\subset \dom(D^*)$.
To prove skew-adjointness of $D$, it therefore suffices to show that
$\dom(D^*)\subset \dom(D)$. Let 
\[
x=\left(\begin{smallmatrix}\tau\\L\\L_b\end{smallmatrix}\right)\in\dom(D^*).
\]
By definition of the adjoint, there exists $x_1\in X$ such that
\begin{equation}\label{eq:Dstar-def}
\langle x, Dx_2\rangle_X=\langle x_1, x_2\rangle_X
\qquad\text{for all }x_2\in\dom(D).
\end{equation}
We show that $x\in\dom(D)$. In the following, we partition $x_1,x_2\in X$ as in \eqref{eq:x1x2}.\\
\emph{(a) Interior regularity.}
Choose $x_2\in\dom(D)$ with $L_{b2}=0$ and such that
$G\tau_2, \rho^{-1}L_2\in C_0^\infty((0,\ell),\R)$.
Then all boundary terms vanish. Using integration by parts in
\eqref{eq:Dstar-def}, we obtain
\begin{align*}
0&= \langle x, Dx_2\rangle_X -\langle x_1, x_2\rangle_X\\
&= \int_0^\ell \Big(
G\tau\,\Big(-\tfrac{\partial}{\partial\xi}(\rho^{-1}L_2)\Big)
+\rho^{-1}L\,\Big(-\tfrac{\partial}{\partial\xi}(G\tau_2)\Big)
\Big)\,d\xi -
\int_0^\ell \Big(G\tau_1\,\tau_2+\rho^{-1}L_1\,L_2\Big)\,d\xi \\
&= \int_0^\ell \Big(
\tfrac{\partial}{\partial\xi}(G\tau)\,\rho^{-1}L_2
+\tfrac{\partial}{\partial\xi}(\rho^{-1}L)\,G\tau_2
\Big)\,d\xi
-
\int_0^\ell \Big(G\tau_1\,\tau_2+\rho^{-1}L_1\,L_2\Big)\,d\xi .
\end{align*}
Since $G\tau_2$ and $\rho^{-1}L_2$ can be chosen arbitrarily and independently in
$C_0^\infty((0,\ell),\R)$, it follows (in the sense of distributions) that $\rho^{-1}L, G\tau\in W^{1,2}([0,\ell],\R)$
with
\begin{equation}\label{eq:bulk-identification-step2}
\tau_1=\tfrac{\partial}{\partial\xi}(\rho^{-1}L),
\qquad
L_1=\tfrac{\partial}{\partial\xi}(G\tau),
\end{equation}
\emph{(b) Boundary conditions.}
Now let $x_2\in\dom(D)$ be arbitrary.
Using \eqref{eq:bulk-identification-step2} and integration by parts, we obtain
\begin{multline*}
\langle x, Dx_2\rangle_X
=\int_0^\ell \Big(
\tfrac{\partial}{\partial\xi}(G\tau)\,\rho^{-1}L_2
+\tfrac{\partial}{\partial\xi}(\rho^{-1}L)\,G\tau_2
\Big)\,d\xi\\
-\Big[ G\tau\,\rho^{-1}L_2+\rho^{-1}L\,G\tau_2\Big]_0^\ell
+\underbrace{J^{-1}\Gamma^{-1}L_b\,
\Gamma(G\tau_2)(\ell)}_{=\,J^{-1}L_b(G\tau_2)(\ell)},\end{multline*}
whereas
\[
\langle x_1,x_2\rangle_X
=\int_0^\ell \Big(G\tau_1\,\tau_2+\rho^{-1}L_1\,L_2\Big)\,d\xi
+J^{-1}\Gamma^{-1}L_{b1}\,L_{b2}.
\]
Using \eqref{eq:bulk-identification-step2}, the integrals coincide, and
\eqref{eq:Dstar-def} reduces to the boundary identity
\begin{equation}\label{eq:boundary-identity-step2}
-\Big[ G\tau\,\rho^{-1}L_2+\rho^{-1}L\,G\tau_2\Big]_0^\ell
+J^{-1}L_b\,(G\tau_2)(\ell)
=
J^{-1}\Gamma^{-1}L_{b1}\,L_{b2}
\qquad\text{for all }x_2\in\dom(D).
\end{equation}
\emph{(b1) Condition $(G\tau)(0)=0$.}
Choose $x_2\in\dom(D)$ with $\tau_2\equiv 0$ and $L_{b2}=0$, and let
$\rho^{-1}L_2\in W^{1,2}([0,\ell],\R)$ satisfy
$(\rho^{-1}L_2)(\ell)=0$ and $(\rho^{-1}L_2)(0)=1$.
Then \eqref{eq:boundary-identity-step2} yields
\[
    (G\tau)(0)=0.
\]
\emph{(b2) Condition $(\rho^{-1}L)(\ell)=J^{-1}L_b$.}
Choose $x_2\in\dom(D)$ with $L_2\equiv 0$ and $L_{b2}=0$, and let
$G\tau_2\in W^{1,2}([0,\ell],\R)$ satisfy $(G\tau_2)(0)=0$ and $(G\tau_2)(\ell)=1$.
Then \eqref{eq:boundary-identity-step2} gives
\[(\rho^{-1}L)(\ell)=J^{-1}L_b.\]
\emph{(c) Conclusion.}
From (a) and (b) we obtain that $x\in \dom(D)$. Altogether, this shows that $\dom(D^*)\subset\dom(D)$. Since, in Step~1, we have shown that any $x\in\dom(D)$ fulfills $x\in\dom(D^*)$ with $D^*x=-Dx$, we obtain that $D^*=-D$.
\item Maximal monotonicity of $B$ follows from its product structure. 
More precisely, $B$ can be written as the Cartesian product of
the zero relation $\setdef{(\tau,0)}{\tau\in L^2([0,\ell],\R)}$,
the Nemytskii relation $\widetilde F_d : L^2([0,\ell])\rightrightarrows L^2([0,\ell])$ defined in \eqref{eq:Nemytskii},
and the relation $F_e:\R\rightrightarrows \R$.\\
The zero relation is trivially maximal monotone.
Maximal monotonicity of $\widetilde F_d$ follows from Proposition~\ref{prop:distrdamp}.
The additional factor $\rho^{-1}$ appearing in the argument of
$\widetilde F_d$ does not affect maximal monotonicity with respect to
the weighted inner product on $X$. Indeed, writing
$v_i=\rho^{-1}L_i$, the corresponding monotonicity inequality becomes
\[
\int_0^\ell
(v_1-v_2)(F_1-F_2)\,d\xi\ge0,
\]
and maximality follows likewise from maximality of $\widetilde F_d$.
The same argument, with $v_i=J^{-1}L_{bi}$, applies to the third
component and $F_e$.
Since the Cartesian product of maximal monotone relations is again maximal monotone,
it follows that $B$ is maximal monotone.\\
Furthermore, Proposition~\ref{prop:distrdamp} yields
\[
\dom(B)=L^2([0,\ell],\R)^2\times J\,\dom(F_e).
\]
As $J>0$ and $0\in\operatorname{int}\dom(F_e)$ by
Assumption~\ref{assump}, we conclude that
$0\in\operatorname{int}\dom(B)$.
\item We choose 
\[
p:=\left(\begin{smallmatrix}G^{-1}\\0\\0\end{smallmatrix}\right)
\in L^\infty([0,\ell],\R)\times\{0\}\times\{0\}\subset X.
\]
It can be immediately seen that $p$ has the desired properties.
\end{enumerate}

It remains to verify the initial-value condition in
Proposition~\ref{prop:specbcs}\,(b).
Let
\[
x_0=
\begin{pmatrix}
\tau_0\\ L_0\\ L_{b0}
\end{pmatrix}
\in X,
\qquad
J^{-1}L_{b0}\in\dom(F_e).
\]
Since functions in $W^{1,2}([0,\ell],\R)$ with a prescribed value at
one endpoint are dense in $L^2([0,\ell],\R)$, there exist
$q_n,r_n\in W^{1,2}([0,\ell],\R)$ such that
\[
q_n\to G\tau_0,\qquad
r_n\to\rho^{-1}L_0
\quad\text{in }L^2([0,\ell],\R),
\]
and
\[
q_n(0)=u(0),\qquad
r_n(\ell)=J^{-1}L_{b0}.
\]
Define
\[
\tau_n:=G^{-1}q_n,\qquad
L_n:=\rho r_n,
\qquad
x_n:=
\begin{pmatrix}
\tau_n\\L_n\\L_{b0}
\end{pmatrix}.
\]
Then
\[
x_n\in\dom(D_f)\cap\dom(B),
\qquad
\mathfrak px_n=u(0),
\]
and, by the boundedness of $G^{-1}$ and $\rho$,
\[
x_n\to x_0\qquad\text{in }X.
\]
Thus
\[
x_0\in
\overline{
\setdef{x\in\dom(D_f)\cap\dom(B)}
{\mathfrak px=u(0)}
},
\]
as required by Proposition~\ref{prop:specbcs}\,(b).
\qedhere

\end{proof}

In the closed-loop system considered in Section~\ref{Sec:ContrDes}, the feedback of the
measured output $z$ turns the boundary condition at $\xi=0$ into one of impedance type.
The preceding result carries over to this case.
\begin{cor}\label{cor:impedance}
Let Assumption~\ref{assump} hold and let $\gamma>0$. Then the assertions of
Theorem~\ref{thm:open_loop_existence} remain true if the boundary condition $(G\tau)(0,t)=u(t)$
in~\eqref{absfor1} is replaced by
\[
  (G\tau)(0,t) - \gamma\,(\rho^{-1}L)(0,t) = v(t),
\]
with $v$ in place of $u$, and with the compatibility condition
$(G\tau_0)(0)-\gamma(\rho^{-1}L_0)(0)=v(0)$ in place of $(G\tau_0)(0)=u(0)$ in statement~(a) of Theorem~\ref{thm:open_loop_existence}.
\end{cor}
\begin{proof}
Set $\tilde p:\operatorname{dom}(D_f)\to\mathbb{R},\ x \mapsto (G\tau)(0)-\gamma(\rho^{-1}L)(0)$,
which is linear and surjective. By Proposition~\ref{prop:specbcs}, it suffices to verify
that $\tilde D:=D_f|_{\ker\tilde p}$ is maximal monotone; conditions (ii) and (iii) in the
proof of Theorem~\ref{thm:open_loop_existence} are unaffected, since $B$ does not depend on the
boundary condition and the element $p=(G^{-1},0,0)^\top$ satisfies $\tilde p\,p=1$
because its second component vanishes.\\
Repeating the computation in Step~1 of the proof of Theorem~\ref{thm:open_loop_existence}~(i)
with $x_1=x_2=x\in\operatorname{dom}(\tilde D)$, the boundary terms at $\xi=\ell$ cancel
as before, while at $\xi=0$ we obtain
\[
  \langle \tilde D x,x\rangle_X = (G\tau)(0)\,(\rho^{-1}L)(0) = \gamma\,\big((\rho^{-1}L)(0)\big)^2\ \ge\ 0,
\]
hence $\tilde D$ is monotone. Repeating Step~2 of that proof, parts (a) and (b) show
that any $x\in\operatorname{dom}(\tilde D^*)$ satisfies $G\tau,\rho^{-1}L\in W^{1,2}([0,\ell],\mathbb{R})$,
$(\rho^{-1}L)(\ell)=J^{-1}L_b$ and $\tilde D^*x=-D_f x$; the boundary identity
corresponding to~\eqref{eq:boundary-identity-step2}, tested with $x_2\in\operatorname{dom}(\tilde D)$
having arbitrary value $(\rho^{-1}L_2)(0)$ and $(G\tau_2)(0)=\gamma(\rho^{-1}L_2)(0)$,
now yields $(G\tau)(0)=-\gamma(\rho^{-1}L)(0)$. Consequently,
\[
  \langle \tilde D^* x,x\rangle_X = -(G\tau)(0)\,(\rho^{-1}L)(0) = \gamma\,\big((\rho^{-1}L)(0)\big)^2\ \ge\ 0,
\]
so $\tilde D^*$ is monotone as well. Since $\tilde D$ is densely defined and closed
(which is verified exactly as for $D$), $-\tilde D$ generates a contraction semigroup
by~\cite[Ch.~1, Cor.~4.4]{Pazy83}; in particular $\operatorname{R}(I+\tilde D)=X$, and
Minty's theorem (Theorem~\ref{thm:Minty}) shows that $\tilde D$ is maximal monotone.
\end{proof}

\subsection{Solution theory of the related wave equation}

In this section, we present an explicit solution to the drill string model \eqref{drillstring}, as we need a dynamic equation describing the evolution of the output $y(t)$ in the analysis of the control design. We emphasize that the implementation of the controller does not require an explicit solution; it utilizes only the measured output~$z(t)$ from the drill string model. 

For controller design, we restrict ourselves to the case where the physical parameters of the drill string \eqref{drillstring}, namely $\rho$ and $G$, are constant and the damping parameters $F_d$ and $F_e$ are single-valued functions. The succeeding considerations are based on the following theorem for the solution of an inhomogeneous wave equation.

\begin{thm}[adapted from \cite{folland2020introduction}, Theorem (5.25)]\label{thm:non-homogeneous wave}
Consider the initial-value problem for the nonhomogeneous wave equation
\begin{equation}\label{eq:wave}
    \phi_{tt}(\xi,t) - c^2 \phi_{\xi\xi}(\xi,t) = h(\xi,t), \quad \xi \in (0,\ell), \quad t > 0,
\end{equation}
with initial conditions
\begin{equation}\label{eq:wave-ic}
    \phi(\xi,0) = \phi_0(\xi), \quad \phi_t(\xi,0) = v_0(\xi),
\end{equation}
where $c>0$, $h\in C^2(\mathbb{R} \times \mathbb{R}_{\ge 0},\R)$, $\phi_0 \in C^2([0,\ell],\R)$, and $v_0 \in C^1([0,\ell],\R)$. Then, the solution to~\eqref{eq:wave},~\eqref{eq:wave-ic} is given by
\begin{equation}\label{sol:non-hom}
    \phi(\xi,t) = \frac{1}{2} \big( \phi_0(\xi - ct) + \phi_0(\xi + ct) \big) + \frac{1}{2c} \int_{\xi - ct}^{\xi + ct} v_0(s) ds+\frac{1}{2c}\int_0^t\int_{\xi-c(t-\tau)}^{\xi+c(t-\tau)}h(s,\tau)ds\,d\tau.
\end{equation}
\end{thm}

\begin{rem}\label{smoothness of the solution}
If $c^2=\frac{G}{\rho}$ for constant $\rho(\cdot)$ and $G(\cdot)$, and $h(\xi,t)=-\tfrac{1}{\rho} F_d(\xi,\phi_t(\xi,t))$,  $\xi\in[0,\ell]$ and $t\geq 0$, for locally integrable $F_d:\R^2\to\R$, then the wave equation~\eqref{eq:wave} in Theorem~\ref{thm:non-homogeneous wave} is the same as the PDE in~\eqref{drillstring} (without the boundary conditions). 
If $\phi_0\in C^2([0,\ell],\R)$ and $v_0\in C^1([0,\ell],\R)$, then $\phi\in C^2([0,\ell]\times \mathbb{R}_{\ge 0},\R)$ satisfies the wave equation in the classical sense. If $\phi_0$, $v_0$ are only locally integrable, then~$\phi$ satisfies the wave equation in the sense of distributions~\cite{evans2022partial,folland2020introduction}.
Furthermore, assume that $F_d(\cdot,v)$ is a twice differentiable function for every fixed~$v\in\R$. Then, iterating Theorem $(6.2)$ of \cite{loomis1968advanced} (chapter 3), we get that $F_d(\xi,\phi_t(\xi,t))$ is twice differentiable as well. Set $\tilde F_d := \tfrac{1}{\rho} F_d$, then the expression \eqref{sol:non-hom} with $h(\xi,t)$ being replaced with $-\tilde F_d(\xi,\phi_t(\xi,t))$ is the solution of the PDE equation in \eqref{drillstring} and can be expressed as 
\begin{align}\label{sol with d}
 \phi(\xi,t)= & \frac{1}{2} \big( \phi_0(\xi - ct) + \phi_0(\xi + ct) \big) + \frac{1}{2c} \int_{\xi - ct}^{\xi + ct} v_0(s) ds -\frac{1}{2c}\int_0^t\int_{\xi-c(t-\tau)}^{\xi+c(t-\tau)}\tilde F_d(s,\phi_t(s,\tau))ds\,d\tau. \end{align}
\end{rem}

It is worthwhile mentioning here that in the next section, we deal with the controller design and it's feasibility. Although the novel funnel controller which we introduce will only require the measured output signal $z$ introduced in \eqref{drillstring} which we assume to be achieved through measurements, we still require a dynamical  expression for  the output signal $y$, which will be then exploited in the proof of the feasibility of the controller. Upon computation, we get from~\eqref{sol with d} that
\begin{align}\label{eq:rep-y(t)}
                   y(t) = \phi_t(\ell,t) &= c(f_1(\ell+ct) + f_2(\ell-ct)) + \int_{0}^{t} (f_3 (\ell+ct, \tau) + f_4 (\ell-ct, \tau))\,d\tau,\\
               \label{eq:rep-z(t)}
                    z(t) = \phi_t(0,t) &= c(f_1(ct) + f_2(-ct)) + \int_{0}^{t} (f_3 (ct, \tau) + f_4 (-ct, \tau))\,d\tau,\\
                    \label{eq:rep-u(t)}
                    u(t) = G\phi_\xi(0,t) &= G\left[f_1(ct) -f_2(-ct) + \frac{1}{c}\int_{0}^{t} \big(f_3 (ct, \tau)- f_4 (-ct, \tau)\big)\,d\tau\right],\\
                \label{eq:rep-phi-xi-l-t}
                    \phi_\xi(\ell,t) &= f_1(\ell+ct) - f_2(\ell-ct) + \frac{1}{c}\int_{0}^{t} \big(f_3 (\ell+ct, \tau)- f_4 (\ell-ct, \tau)\big)\,d\tau,
                \end{align}    
where\begin{align*}
f_1(\xi + ct) &= \frac{1}{2}\left(\phi_0' (\xi +ct) + \frac{1}{c}v_0(\xi +ct)\right), \\
f_2(\xi - ct) &= \frac{1}{2}\left(\frac{1}{c}v_0(\xi -ct) - \phi_0'(\xi -ct)\right), \\
f_3(\xi + ct, \tau) &=- \frac{1}{2}\tilde F_d(\xi +c(t-\tau),\phi_t(\xi+c(t-\tau),\tau)), \\
f_4(\xi - ct, \tau) &= -\frac{1}{2}\tilde F_d(\xi -c(t-\tau),\phi_t(\xi-c(t-\tau),\tau)).
\end{align*}
Combining the equations \eqref{eq:rep-y(t)}, \eqref{eq:rep-z(t)},\eqref{eq:rep-u(t)}, \eqref{eq:rep-phi-xi-l-t} with the second equation of \eqref{drillstring}, we get 
 \begin{equation}\label{eq:y-delaydiff}
                    J\dot{y}(t) = \frac{-G\Gamma}{c}\Bigg(y(t)+\frac{c}{G}u\left(t-\tfrac{l}{c}\right)-z\left(t-\tfrac{l}{c}\right)+\int_{t-\frac{l}{c}}^{t}  \tilde F_d(\ell-c(t-\tau),\phi_t(\ell-c(t-\tau),\tau))d\tau\Bigg) -F_e(y(t)).
                \end{equation}

\section{Controller design}\label{Sec:ContrDes}

In this section, we introduce a novel funnel controller for the output tracking of the considered drill string model \eqref{drillstring}. We restrict ourselves to the case of constant physical parameters $\rho(\cdot)$ and $G(\cdot)$ and single-valued damping functions $F_d$, $F_e$. Set $c^2=\frac{G}{\rho}$ and $\omega := \ell/c$. We assume availability of the instantaneous signal $z(t)$. 

In contrast to Section~\ref{Sec:wellpose}, the above restriction is not
needed for well-posedness but for the controller design and its feasibility analysis:
the design compensates the wave traveling time $\omega$ by processing the delayed output
$y(t-\omega)$, and the analysis rests on the delay differential
equation~\eqref{eq:y-delaydiff} for the bit velocity. For variable coefficients no such explicit
representation is available in general, and the extension of the delay-compensation
mechanism to this case is left for future work. Well-posedness of the resulting
closed-loop system is nevertheless established within the framework of
Section~\ref{Sec:wellpose}, see Corollary~\ref{cor:impedance} and
the proof of Theorem~\ref{Thm:funnel} below.

For a given reference signal $y_{\rm ref}\in W^{1,\infty}([-\omega,\infty),\R)$ the novel funnel control design is given by
\begin{equation}\label{eq:FC}
\boxed{
\begin{aligned}
    e(t) &= \frac{y(t-\omega) - y_{\rm ref}(t-\omega) + I(t)}{\psi(t-\omega)},\\
    v(t) &= k\, \frac{e(t)}{1-e(t)^2} + \left(\hat v - k\, \frac{e(0)}{1-e(0)^2}\right) p(t),\\
    \dot I(t) &= -\alpha I(t) - \beta\big(v(t) - v(t-2\omega)\big),\quad I(0) = 0,\\
    u(t) &= \frac{\textcolor{blue}{G}}{c} z(t) + v(t),\\
\end{aligned}
}
\end{equation}
with the design parameters
\[
    \psi\in \Psi,\ \hat v\in\R,\ k>0,\ \alpha>0,\ \beta>0,
\]
and $p$ a smooth function with compact support and $p(0)=1$,
where
\[
    \Psi := \setdef{\psi\in W^{1,\infty}([-\omega,\infty),\R)}{\inf_{t\ge -\omega} \psi(t) > 0}.
\]
Furthermore, we assume that $v(t) = \hat v$ for $t\in [-2\omega,0]$. The value of $\hat v$ is determined so that the second of equations~\eqref{drillstring} is satisfied at $t=0$, that is 
\[
    \hat v := G \phi_\xi(0,0) - \frac{\textcolor{blue}{G}}{c} v_0(0).
\]
The role of the function $p$ is to guarantee this initial condition, while its effect vanishes after a finite time and it can be seen as a bounded perturbation. Observe that the controller requires knowledge of $G$, $\rho$ and
$\ell$, but not of the damping laws or of the state of the string. The feasibility result
Theorem~\ref{Thm:funnel} moreover fixes the design parameters $\alpha,\beta$ in terms
of $G$, $\Gamma$, $J$ and $c$.  

\begin{figure}[!ht]
\begin{center}
\begin{tikzpicture}[scale=0.45]
\tikzset{>=latex}
  \filldraw[color=gray!25] plot[smooth] coordinates {(0.15,4.7)(0.7,2.9)(4,0.4)(6,1.5)(9.5,0.4)(10,0.333)(10.01,0.331)(10.041,0.3) (10.041,-0.3)(10.01,-0.331)(10,-0.333)(9.5,-0.4)(6,-1.5)(4,-0.4)(0.7,-2.9)(0.15,-4.7)};
  \draw[thick] plot[smooth] coordinates {(0.15,4.7)(0.7,2.9)(4,0.4)(6,1.5)(9.5,0.4)(10,0.333)(10.01,0.331)(10.041,0.3)};
  \draw[thick] plot[smooth] coordinates {(10.041,-0.3)(10.01,-0.331)(10,-0.333)(9.5,-0.4)(6,-1.5)(4,-0.4)(0.7,-2.9)(0.15,-4.7)};
  \draw[thick,fill=lightgray] (0,0) ellipse (0.4 and 5);
  \draw[thick] (0,0) ellipse (0.1 and 0.333);
  \draw[thick,fill=gray!25] (10.041,0) ellipse (0.1 and 0.333);
  \draw[thick] plot[smooth] coordinates {(0,2)(2,1.1)(4,-0.1)(6,-0.7)(9,0.25)(10,0.15)};
  \draw[thick,->] (-2,0)--(12,0) node[right,above]{\normalsize$t$};
  \draw[thick,dashed](0,0.333)--(10,0.333);
  \draw[thick,dashed](0,-0.333)--(10,-0.333);
  \node [black] at (0,2) {\textbullet};
  \draw[->,thick](4,-3)node[right]{\normalsize$\lambda$}--(2.5,-0.4);
  \draw[->,thick](3,3)node[right]{\normalsize$(0,w(0))$}--(0.07,2.07);
  \draw[->,thick](9,3)node[right]{\normalsize$\psi(t)$}--(7,1.4);
\end{tikzpicture}
\end{center}
\caption{Error evolution in a funnel $\mathcal F_{\psi}$ with boundary $\psi(t)$.}
\label{Fig:funnel}
\end{figure}
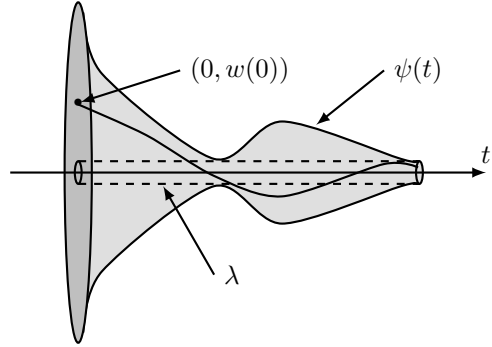

In~\eqref{eq:FC}, $e(t)$ represents the normalized tracking error and $\psi(t)$ is a positive function that ensures that the corrected tracking error $w(t):= y(t) - y_{\rm ref}(t) + I(t+\omega)$ evolves in a prescribed performance funnel
\[
    \mathcal F_{\psi} = \setdef{(t,w)\in [-\omega,\infty)\times\R}{ |w| < \psi(t)},
\]
see Fig.~\ref{Fig:funnel}. By the properties of~$\Psi$ there exists $\lambda>0$ such that $\psi(t)\ge \lambda$ for all $t\ge 0$. It is important to note that the function $\psi\in\Psi$ is a user-defined design parameter in the controller~\eqref{eq:FC}. Although~$\psi$ does not need to be monotonically decreasing in general, it is usually convenient to choose it of the form $\psi(t) = a e^{-b t} + d$ with positive parameters $a,b,d$. Other typical choices for funnel boundaries are outlined in~\cite[Sec.~3.2]{Ilchmann2012DecentralizedTO}.

$I(t)$ is an auxiliary variable introduced as a correction term for the tracking error, necessary due to the wave traveling time (treated as a delay) of magnitude $\omega$ in the dynamics~\eqref{eq:y-delaydiff}. The third of equations~\eqref{eq:FC} dynamically adjusts the delay-dependent correction term to actively compensate for the delayed influence of input $u(t)$ on the to be controlled output $y(t)$. Such a term has been first used in~\cite{BergBika25}, which inspired the design~\eqref{eq:FC}. Since $\psi \in \Psi$, it is always positive and bounded away from zero, stressing that asymptotic tracking is not pursued here. The function $v(t)=k\frac{e(t)}{1-e(t)^2}$ transforms the tracking error $e(t)$ into the control-related variable $v(t)$ and $k$ is a constant parameter. The denominator in~$v(t)$ ensures that $e(t)$ stays within $(-1,1)$ for proper operation, meaning the controller guarantees the evolution of the corrected error~$w$ within the funnel $\mathcal F_{\psi}$. The term involving $\beta$ in the ODE for~$I$ introduces an effect based on past values of $v(t)$, helping to shape the control response over time.

The following assumptions are made on the drill string system parameters, the reference signal and controller parameters.

\begin{assumption}\label{assump for the controller}\
 \begin{enumerate}[label=(\roman*),ref=(\roman*)]
    \item $y_{\rm ref}\in W^{1,\infty}([-\omega,\infty),\mathbb{R})$, and
$F_e:\mathbb{R}\to\mathbb{R}$ is continuous, monotone and bounded.
    \item The parameters $\rho(\cdot)$ and $G(\cdot)$ in~\eqref{drillstring} are constant and the initial conditions satisfy $\phi_0 \in C^2([0,\ell],\R)$ and $v_0 \in C^1([0,\ell],\R)$.
    \item  $F_d\in C^2(\mathbb{R}^2,\mathbb{R})$ is bounded, and $v\mapsto F_d(\xi,v)$ is monotone for all
$\xi\in\mathbb{R}$.
     \item 
     the initial history $y|_{[-\omega,0]} = y_{\rm hist}\in C^1([-\omega,0],\mathbb{R})$
is compatible with the initial data, $y_{\rm hist}(0)=v_0(\ell)$, and lies within the
performance funnel, i.e., $|y_{\rm hist}(t)-y_{\rm ref}(t)|<\psi(t)$ for all
$t\in[-\omega,0]$.
\end{enumerate}   
\end{assumption}


Regarding the assumption on the initial history, note that due to the wave traveling time $\omega$ the
closed-loop system is a delay system and hence its initial
data consist of an initial history rather than a single initial value. Condition (iv)
is the delay counterpart of the standard requirement in funnel control that the initial
error lies inside the funnel, $|e(0)|<1$; indeed, since $I(0)=0$ we have
$e(0) = \big(y_{\rm hist}(-\omega)-y_{\rm ref}(-\omega)\big)/\psi(-\omega)$.
Usually the controller is switched on when the system is at rest, so that
$y_{\rm hist}\equiv 0$ and $v_0\equiv 0$, so the compatibility condition of~(iv) holds trivially and
the funnel condition reduces to $|y_{\rm ref}(t)|<\psi(t)$ on $[-\omega,0]$.

\begin{rem}
A straightforward verification yields that under Assumption~\ref{assump for the controller}, Assumption~\ref{assump} is satisfied as well. A continuous monotone function $\mathbb{R}\to\mathbb{R}$ is maximal monotone, so that
$F_e$ and $F_d(\xi,\cdot)$ are maximal monotone with full domain. The differentiability of $\xi\mapsto F_d(\xi,v)$, $v\in\R$, implies the measurability required in
Assumption~\ref{assump}\,\ref{assumpb}\,\ref{assumpb3}. The growth condition on $F_d$ is satisfied due to boundedness. 
\end{rem}

Next, we analyze feasibility of the proposed control design \eqref{eq:FC}. Note that by~\eqref{eq:rep-y(t)}--\eqref{eq:rep-u(t)} the unknown value of $y(t-\omega)$ can be computed as 
\begin{align*}
      y(t-\omega) = \frac12 \Bigg( \frac{c}{G}u(t) - \frac{c}{G}u(t-2\omega) + z(t) + z(t-2\omega) + \int_{t-\omega}^t  \tilde F_d(c(t-s),\phi_t(c(t-s),s) {\rm d}s \\- \int_{t-2\omega}^{t-\omega} \tilde F_d(2\ell-c(t-s),\phi_t(2\ell-c(t-s),s) {\rm d}s  \bigg ).
\end{align*}
In case of absence of damping ($F_d=0$), the value of $e(t)$ in~\eqref{eq:FC} can hence be determined by inserting the above relation into the first equation of~\eqref{eq:FC} and solving for $e(t)$. This is indeed always possible (e.g.\ using Newton's method), since we have
\begin{align*}
    & \psi(t-\omega) e(t) = y(t-\omega) - y_{\rm ref}(t-\omega) + I(t)\\
    \iff\quad & \psi(t-\omega) e(t) =  \frac12 \Big( \frac{c}{G}v(t) - \frac{c}{G}v(t-2\omega) + 2z(t)\Big) - y_{\rm ref}(t-\omega) + I(t)\\
    \iff\quad & \psi(t-\omega) e(t) - \frac{ck}{2G} \frac{e(t)}{1-e(t)^2} = -\frac{c}{2G} v(t-2\omega) + z(t) - y_{\rm ref}(t-\omega) + I(t)
\end{align*}
In order to show that the above equation is always solvable for $e(t)$ with $|e(t)|<1$, consider the function 
\[
    F_t:(-1,1)\to\R,\ e\mapsto \psi(t-\omega) e - \frac{ck}{2G} \frac{e}{1-e^2},
\]
which has derivative 
\[
    F_t'(e) = \psi(t-\omega) - \frac{ck}{2G} \frac{1+e^2}{(1-e^2)^2} \le \psi(t-\omega) - \frac{ck}{2G} < 0
\]
for 
\[
    k > \frac{2G\|\psi\|_\infty}{c}.
\]
Under this condition, the controller equations~\eqref{eq:FC} can always be uniquely solved for~$e(t)$. Now in case of bounded damping (by Assumption~\ref{assump for the controller}), we have that
\[
    |\tilde F_d(\cdot,\cdot)| \leq F_{d_{\max}},\quad F_{d_{\max}}>0.
\]
This leads to the modified equation
\[
    F_t(e(t)) = R(t) + \Delta(t),
\]
with $F_t$ as above, known
\[
    R(t) := -\frac{c}{2G} v(t-2\omega) + z(t) - y_{\rm ref}(t-\omega) + I(t)
\]
and \textit{unknown} damping-induced perturbation $\Delta(t)$, which satisfies
\[
    |\Delta(t)| \leq 2\omega F_{d_{\max}}.
\]
For controller implementation, $e(t)$ must again be computed as the solution of $F_t(e(t)) = R(t)$ since $\Delta(t)$ is unknown. The deviation of $e(t)$ from the ``real error'' $\tilde{e}(t)$ satisfying $F_t(\tilde e(t)) = R(t) + \Delta(t)$ is small whenever the magnitude $F_{d_{\max}}$ of the distributed damping is small or the parameter~$k$ is large: Applying the mean value theorem we obtain
\[
    |e(t) - \tilde{e}(t)| \leq \frac{|\Delta(t)|}{\min_{e \in (-1,1)} |F_t'(e)|}\le \frac{2\omega F_{d_{\max}}}{\tfrac{ck}{2G}-\|\psi\|_\infty}.
\]
In this case, $e(t)$ approximates the true error $\tilde e(t)$ with the above quality. In practice this effect will be additionally superposed by measurement noise. Thus, if the term $\tfrac{2\omega F_{d_{\max}}}{\tfrac{ck}{2G}-\|\psi\|_\infty}$ is of the same magnitude as the expected measurement noise, then the approximation error $ |e(t) - \tilde{e}(t)|$ cannot be distinguished from this noise.
In the following we present the main result on the feasibility of the controller. It is important to stress that application of the controller \eqref{eq:FC} to the model \eqref{drillstring} leads to a closed-loop system given by

\begin{equation}\label{eq:closed loop}
\begin{aligned}
\phi_{tt}(\xi,t)&=c^2 \phi_{\xi\xi}(\xi,t)- \tilde F_d(\xi,\phi_t(\xi,t)),\\
     J \dot{y}(t) &= \frac{-G\Gamma}{c} \bigg( y(t) + \frac{c}{G}v(t - \omega) 
    + \int_{t - \omega}^{t} \tilde F_d(\ell - c(t - s),\, \phi_t(\ell - c(t - s), s))\, ds \bigg) - F_e(y(t)), \\
    \dot{I}(t) &= -\alpha I(t) - \beta\big(v(t) - v(t - 2\omega)\big), \\
    v(t) &= \frac{ke(t)}{1 - e(t)^2} + \left(\hat v - k\, \frac{e(0)}{1-e(0)^2}\right) p(t),\quad e(t) = \frac{y(t-\omega) - y_{\rm ref}(t-\omega) + I(t)}{\psi(t-\omega)},\\
    G\phi_{\xi}(0,t)&=\frac{G}{c} \phi_t(0,t) + v(t),\quad   \phi_t(\ell,t) =y(t),\\
    \phi(\xi,0)&=\phi_0(\xi),\quad \phi_t(\xi,0)=v_0(\xi),\quad y|_{[-\omega,0]} = y_{\rm hist},\quad I(0)=0,
\end{aligned}
\end{equation}
where $\tilde F_d = \tfrac{1}{\rho} F_d$. For $t\in[0,\omega)$, the term $\frac{c}{G} v(t-\omega)$ in the second equation
of~\eqref{eq:closed loop} has to be read as $-r_0(t)$ and the lower integration limit
$t-\omega$ as $0$, where 
\begin{equation}\label{eq:r0}
    r_0(t):=v_0(\ell-ct)-c\,\phi_0'(\ell-ct)
\end{equation}
is the wave
arriving at $\xi=\ell$ from the initial data, see Step~3 of the proof of
the main result in Theorem~\ref{Thm:funnel}. We prepend its statement with the following lemma that contains the core funnel argument, which will be applied on each interval of length $\omega$ in the proof of the theorem.
\begin{lemma}\label{lem:funnel-ode}
Let $t_0<t_1$, $\kappa>0$, and let $\theta\in W^{1,\infty}([t_0,t_1],\mathbb{R})$ with
$\inf_{t\in[t_0,t_1]}\theta(t)>0$. Let $M:[t_0,t_1]\times(-1,1)\to\mathbb{R}$ be
measurable in $t$, locally Lipschitz continuous in $e$ uniformly in $t$, and bounded,
i.e., $|M(t,e)|\le\bar M$ for all $(t,e)\in[t_0,t_1]\times(-1,1)$ and some $\bar M>0$.
Choose $\varepsilon_0\in(0,1)$ such that
\begin{equation}\label{eq:eps0}
  \frac{\kappa\,\varepsilon_0^2}{1-\varepsilon_0^2}\ \ge\ \bar M+\|\dot\theta\|_{\infty}.
\end{equation}
Then, for every $e_0\in(-1,1)$, the initial value problem
\begin{equation}\label{eq:funnel-ode}
  \dot e(t)=\frac{1}{\theta(t)}\Big(M\big(t,e(t)\big)-\kappa\,\frac{e(t)}{1-e(t)^2}-e(t)\dot\theta(t)\Big),
  \qquad e(t_0)=e_0,
\end{equation}
has a unique solution $e\in W^{1,\infty}([t_0,t_1],\mathbb{R})$, and it satisfies
\[
  |e(t)|\ \le\ \max\{|e_0|,\varepsilon_0\}\ <\ 1\qquad\text{for all } t\in[t_0,t_1].
\]
\end{lemma}
\begin{proof}
Since $\theta$ is bounded away from zero, the right-hand side of~\eqref{eq:funnel-ode} is
measurable in $t$, locally Lipschitz continuous in $e\in(-1,1)$ uniformly in $t$, and
bounded on $[t_0,t_1]\times[-r,r]$ for every $r\in(0,1)$. By the standard theory of
Carath\'eodory differential equations (see, e.g., \cite[\S10]{Walter1998}), there exists a
unique maximal solution $e:[t_0,\bar\kappa)\to(-1,1)$ of~\eqref{eq:funnel-ode} with
$\bar\kappa\in(t_0,t_1]$, and if $\bar\kappa<t_1$, then $e$ leaves every compact subset of
$(-1,1)$ as $t\to\bar\kappa$.\\
Set $\varepsilon:=\max\{|e_0|,\varepsilon_0\}\in(0,1)$. For almost all $t\in[t_0,\bar\kappa)$
we compute, using $|e(t)|<1$,
\[
  \tfrac12\ddt e(t)^2
  = \frac{1}{\theta(t)}\Big(e(t)M\big(t,e(t)\big)-\kappa\frac{e(t)^2}{1-e(t)^2}-e(t)^2\dot\theta(t)\Big)
  \le \frac{1}{\theta(t)}\Big(\bar M+\|\dot\theta\|_{\infty}-\kappa\frac{e(t)^2}{1-e(t)^2}\Big).
\]
Since $s\mapsto s/(1-s)$ is increasing on $[0,1)$, \eqref{eq:eps0} implies
$\ddt e(t)^2\le 0$ for almost all $t$ with $|e(t)|\ge\varepsilon_0$.
Now suppose that $|e(t_2)|>\varepsilon$ for some $t_2\in[t_0,\bar\kappa)$. As $|e(t_0)|\le\varepsilon$,
the number $t_3:=\max\{\,t\in[t_0,t_2]\,|\,|e(t)|\le\varepsilon\,\}$ is well-defined and
satisfies $|e(t)|>\varepsilon\ge\varepsilon_0$ for all $t\in(t_3,t_2]$. Hence
$e(t_2)^2\le e(t_3)^2=\varepsilon^2$, a contradiction. Therefore $|e(t)|\le\varepsilon<1$ for
all $t\in[t_0,\bar\kappa)$.\\
Consequently, $\dot e$ is bounded on $[t_0,\bar\kappa)$, so $e$ is Lipschitz continuous and
stays in the compact set $[-\varepsilon,\varepsilon]\subset(-1,1)$. By maximality this
yields $\bar\kappa=t_1$, and $e$ extends to a solution on $[t_0,t_1]$ with
$e\in W^{1,\infty}([t_0,t_1],\mathbb{R})$ and $|e(t)|\le\varepsilon$ for all $t\in[t_0,t_1]$.
\end{proof}

We are now ready to state the main result, where we show existence of a global and bounded solution of the closed-loop system~\eqref{eq:closed loop}.

\begin{thm}\label{Thm:funnel}
Consider the drill string system~\eqref{drillstring} under the application of the funnel controller~\eqref{eq:FC}, where \( \alpha = \frac{G\Gamma}{cJ} \) and \textcolor{blue}{\( \beta = \frac{\Gamma}{J} \)}. Assume that Assumption \ref{assump for the controller} holds. Then we have:
\begin{enumerate}[label=(\roman*),ref=(\roman*)]
    \item The closed-loop system~\eqref{eq:closed loop}  has a solution, and every maximal solution is global and bounded with $(\phi,y,I)\in C(\mathbb{R}_{\ge 0},L^2([0,\ell],\R))\times W^{1,\infty}([-\omega,\infty),\R)\times W^{1,\infty}(\R_{\ge 0},\R)$ and, in particular, $\phi_t,\phi_\xi\in C(\mathbb{R}_{\ge 0},L^2([0,\ell],\R))$.
    \item The input~$v$ is bounded.
    \item The system output exhibits a prescribed performance in the sense that, for all $t\ge 0$,
    \[
    |y(t-\omega) - y_{\rm ref}(t-\omega) + I(t)| < \psi(t-\omega).
  \]
\end{enumerate}
\end{thm}
\begin{proof}
The closed-loop system~\eqref{eq:closed loop} is a delay system with delays $\omega$
and $2\omega$. We therefore proceed by the method of steps: on each interval
$[K\omega,(K+1)\omega]$, the controller signal $v$ depends only on $e$ and on
quantities already known from $[(K-2)\omega,K\omega]$, so that first the scalar
controller equation, then the wave equation (via Corollary~\ref{cor:impedance}) and
finally the equation for $y$ can be solved successively; the funnel property
$|e(t)|<1$ is propagated from one interval to the next.

Throughout the proof, let $e(0):=\big(y_{\rm hist}(-\omega)-y_{\rm ref}(-\omega)\big)/\psi(-\omega)$,
which satisfies $|e(0)|<1$ by Assumption~\ref{assump for the controller}~(iv), let
$\hat p(t):=\big(\hat v-k\frac{e(0)}{1-e(0)^2}\big)p(t)$, 
and let $r_0$ be as in~\eqref{eq:r0}. Define
\begin{equation}\label{eq:Mbar}
  \bar M := \alpha\big(\|\psi\|_\infty+\|y_{\rm ref}\|_\infty+\|y_{\rm hist}\|_\infty
            +\|r_0\|_\infty+\omega \|\tilde F_d\|_{\infty}\big)
          + \|\dot y_{\rm ref}\|_\infty+\|\dot y_{\rm hist}\|_\infty
          + \tfrac1J\|F_e\|_\infty + \beta\big(\|\hat p\|_\infty+|\hat v|\big),
\end{equation}
which is finite by Assumption~\ref{assump for the controller}, fix $\varepsilon_0\in(0,1)$ such
that~\eqref{eq:eps0} holds with $\kappa:=\beta k$ and $\|\dot\theta\|_{\infty}$ replaced
by $\|\dot\psi\|_{\infty}$, and set
$\varepsilon:=\max\{|e(0)|,\varepsilon_0\}<1$. Note that $\bar M$, $\varepsilon_0$ and
$\varepsilon$ do not depend on the interval under consideration.

\emph{Step 1 (Controller equation on $[0,\omega]$).}
For $t\in[0,\omega]$ we have $t-\omega\in[-\omega,0]$ and $t-2\omega\in[-2\omega,-\omega]$,
hence $y(t-\omega)=y_{\rm hist}(t-\omega)$ and $v(t-2\omega)=\hat v$. Differentiating the
first equation in~\eqref{eq:FC} and inserting the second and third one, we find
that on $[0,\omega]$ the equations~\eqref{eq:FC} are equivalent to
\begin{equation}\label{eq:step1-ode}
  \dot e(t) = \frac{1}{\psi(t-\omega)}\Big( M_0(t) - \alpha\psi(t-\omega)\,e(t)
     - \beta k\,\frac{e(t)}{1-e(t)^2} - e(t)\dot\psi(t-\omega)\Big),
\end{equation}
with initial value $e(0)$ as above, together with
$I(t)=\psi(t-\omega)e(t)-y_{\rm hist}(t-\omega)+y_{\rm ref}(t-\omega)$, where
\[
  M_0(t) := \dot y_{\rm hist}(t-\omega)-\dot y_{\rm ref}(t-\omega)
   + \alpha\big(y_{\rm hist}(t-\omega)-y_{\rm ref}(t-\omega)\big)-\beta\big(\hat p(t)-\hat v\big).
\]
Equation~\eqref{eq:step1-ode} is of the form~\eqref{eq:funnel-ode} with
$[t_0,t_1]=[0,\omega]$, $\theta(t)=\psi(t-\omega)$, $\kappa=\beta k$ and
$M(t,e)=M_0(t)-\alpha\psi(t-\omega)e$. This $M$ is measurable in $t$, affine in $e$, and
satisfies $|M(t,e)|\le\bar M$ on $[0,\omega]\times(-1,1)$ by~\eqref{eq:Mbar}. Hence
Lemma~\ref{lem:funnel-ode} yields a unique solution
$e\in W^{1,\infty}([0,\omega],\mathbb{R})$ of~\eqref{eq:step1-ode} with $|e(t)|\le\varepsilon$
for all $t\in[0,\omega]$. Consequently, $I,v\in W^{1,\infty}([0,\omega],\mathbb{R})$, and
$v(0)=\hat v$ by $p(0)=1$, consistent with $v=\hat v$ on $[-2\omega,0]$.

\textit{Step 2 (PDE solution).}
Consider the input $u(t) = v(t) + \frac{G}{c}z(t)$, which by~\eqref{drillstring} can be reformulated as the boundary condition $G\phi_\xi(0,t) - \frac{G}{c}\phi_t(0,t) = v(t)$; in the variables of~\eqref{absfor1} this reads
$(G\tau)(0,t)-\tfrac Gc(\rho^{-1}L)(0,t)=v(t)$. Since $v\in W^{1,\infty}([0,\omega],\R)$ by Step~1, $\dom(F_e)=\R$ and $W^{1,2}([0,\ell],\R)$ is dense in $L^2([0,\ell],\R)$ we have that $\overline{\dom(D_f)\cap\dom(B)}=X$, thus it follows from Corollary~\ref{cor:impedance} with $\gamma=\tfrac Gc$ that there exists  a unique generalized solution \( \phi \in C([0, \omega], L^2([0,\ell],\R)) \) with $\phi_t,\phi_\xi \in C([0, \omega], L^2([0,\ell],\R))$.  

\textit{Step 3 (Bit dynamics on $[0,\omega]$).}
For $t\in[0,\omega]$, the wave arriving at $\xi=\ell$ at time $t$ emanates from the
initial data at $\xi=\ell-ct\in[0,\ell]$ rather than from the boundary $\xi=0$. Hence,
in the derivation of~\eqref{eq:y-delaydiff} the term $\tfrac cG u(t-\omega)-z(t-\omega)$, which
represents the wave leaving $\xi=0$ at time $t-\omega$, has to be replaced by $-r_0(t)$ as in~\eqref{eq:r0} and the damping integral has to be taken along the characteristic from time $0$,
\[
  D_\phi^0(t) := \int_0^t \tilde F_d\big(\ell-c(t-s),\phi_t(\ell-c(t-s),s)\big)\,\mathrm{d}s,
\]
which is well-defined by Step~2 and Remark~\ref{smoothness of the solution}. Thus, on $[0,\omega]$ the output satisfies
\[
  J\dot y(t) = -\frac{G\Gamma}{c}\Big(y(t) - r_0(t) + D_\phi^0(t)\Big) - F_e(y(t)),
  \qquad y(0)=y_{\rm hist}(0)=v_0(\ell),
\]
where the initial value is consistent with Assumption~\ref{assump for the controller}~(iv). Since
$r_0$ is continuous by Assumption~\ref{assump for the controller}~(ii) and $F_d$, $F_e$ are bounded,
there exists a solution $y$ of the above differential equation on $[0,\omega]$ by
standard existence results for ODEs combined with the variation of constants formula
for linear ODEs (by which $y$ cannot exhibit blow-up on $[0,\omega]$). It follows that
$y\in W^{1,\infty}([-\omega,\omega],\mathbb{R})$.

\textit{Step 4 (Inductive continuation~-- error dynamics).}
Assume that for some \( K \in \mathbb{N} \), the solution $(\phi, y, I)$ exists on \( [0, K\omega] \) with \( |e(t)| \le \varepsilon \) for all $t \in [0, K\omega]$. We now extend the solution to \( [K\omega, (K+1)\omega] \). In order to do this, we construct a differential equation for the error $e$ and show that it has a solution on \( [K\omega, (K+1)\omega] \). Observe that by~\eqref{eq:FC} we have
\begin{align*}
    \dot e(t) &= \frac{\dot y(t-\omega) - \dot y_{\rm ref}(t-\omega) + \dot I(t)}{\psi(t-\omega)} - \frac{y(t-\omega) - \yref(t-\omega) + I(t)}{\psi(t-\omega)^2} \dot \psi(t-\omega)\\
    &= \frac{1}{\psi(t-\omega)}\bigg(\frac{-G\Gamma}{cJ} \big( y(t-\omega) + \frac{c}{G} v(t - 2\omega) + D_\phi(t-\omega) \big) - \frac{1}{J}F_e(y(t-\omega)) \\
    &\quad - \dot y_{\rm ref}(t-\omega) - \alpha I(t) - \beta(v(t) - v(t-2\omega)) - e(t) \dot \psi(t-\omega)\bigg),
\end{align*}
where \( D_\phi(t) = \int_{t - \omega}^{t} \tilde F_d(\ell - c(t - s), \phi_t(\ell - c(t - s), s))\, {\rm d}s \), which is well-defined by Step~2 and Remark~\ref{smoothness of the solution}.
For $K=1$ we have $t-\omega\in[0,\omega]$, and according to Step~3 the term
$\tfrac{c}{G} v(t-2\omega)$ has to be replaced by $-r_0(t-\omega)$ and $D_\phi(t-\omega)$ by
$D_\phi^0(t-\omega)$; since both are bounded, they can be absorbed into $M(t)$ below
and the argument is unchanged.
Now, substitute $y(t-\omega)=\psi(t-\omega)e(t)+y_{\rm ref}(t-\omega)-I(t)$ and use the
definitions $\alpha=\frac{G\Gamma}{cJ}$ and $\beta=\frac{\Gamma}{J}$, which imply
$\beta=\frac{c}{G}\alpha$, so that the terms containing $v(t-2\omega)$ cancel, thus
\begin{align*}
    \dot e(t) &= \frac{1}{\psi(t-\omega)}\bigg(-\alpha \big( \psi(t-\omega)e(t) + \yref(t-\omega) + D_\phi(t-\omega) \big) - \frac{1}{J}F_e(y(t-\omega)) \\
    &\quad - \dot y_{\rm ref}(t-\omega) - \beta v(t)  - e(t) \dot \psi(t-\omega)\bigg).
\end{align*}
Next, we substitute the control law $v(t)=\frac{ke(t)}{1-e(t)^2}+\hat p(t)$, giving
\begin{align*}
    \dot e(t) &= \frac{1}{\psi(t-\omega)}\bigg(-\alpha \big( \psi(t-\omega)e(t) + \yref(t-\omega) + D_\phi(t-\omega) \big) \blue {- \frac{1}{J}F_e(y(t-\omega))} \\
    &\quad - \dot y_{\rm ref}(t-\omega) - \beta k \frac{e(t)}{1 - e(t)^2} - \beta \hat p(t) - e(t) \dot \psi(t-\omega)\bigg).
\end{align*}
The differential equation for $e$ derived above is of the form~\eqref{eq:funnel-ode} with
$[t_0,t_1]=[K\omega,(K+1)\omega]$, $\theta(t)=\psi(t-\omega)$, $\kappa=\beta k$, initial value
$e_0=e(K\omega)$, and
\[
 M(t,e):=-\alpha\big(\psi(t-\omega)\,e+y_{\rm ref}(t-\omega)+D_\varphi(t-\omega)\big)
 -\tfrac1J F_e\big(y(t-\omega)\big)-\dot y_{\rm ref}(t-\omega)-\beta\,\hat p(t)
\]
(for $K=1$, with the modifications indicated above). Since the signals $y$, $D_\phi$
are already known on $[0,K\omega]$, $M$ is a well-defined function on
$[K\omega,(K+1)\omega]\times(-1,1)$, measurable in $t$ and affine in $e$, and
$|M(t,e)|\le\bar M$ by~\eqref{eq:Mbar}, using $|D_\phi(t-\omega)|\le\omega \|\tilde F_d\|_{\infty}$ and, for $K=1$,
$|r_0|\le\|r_0\|_\infty$ and $|\beta\hat v|\le\beta|\hat v|$. Moreover, $|e(K\omega)|\le\varepsilon$
by the induction hypothesis. Hence Lemma~\ref{lem:funnel-ode} yields a unique solution
$e\in W^{1,\infty}([K\omega,(K+1)\omega],\mathbb{R})$ with
\[
  |e(t)|\ \le\ \max\{|e(K\omega)|,\varepsilon_0\}\ \le\ \varepsilon\qquad
  \text{for all } t\in[K\omega,(K+1)\omega].
\]
\textit{Step 5 (Inductive continuation~-- remaining variables).}
As shown in Step~4, $e$ is well-defined on $[K\omega,(K+1)\omega]$, thus $v$ and $I$ are well-defined on $[K\omega,(K+1)\omega]$, too. Similar to Step~2,  the input is given by $u(t) = v(t) + \frac{G}{c}z(t)$, which can be reformulated as the boundary condition $G\phi_\xi(0,t) - \frac{G}{c}\phi_t(0,t) = v(t)$. By the solution of the previous induction step, this boundary condition is satisfied at $t=K \omega$. Since $v\in W^{1,\infty}([K\omega,(K+1)\omega],\R)$, $\dom(F_e)=\R$ and  $W^{1,2}([0,\ell],\R)$ is dense in $L^2([0,\ell],\R)$ we have that $\overline{\dom(D_f)\cap\dom(B)}=X$,  thus it follows from Corollary~\ref{cor:impedance} (as in Step~2) that there exists a unique generalized solution \( \phi \in C([K\omega, (K+1)\omega], L^2([0,\ell],\R)) \)  with $\phi_t,\phi_\xi \in C([K\omega, (K+1)\omega], L^2([0,\ell],\R))$. Then $y$ is determined by the differential equation 
\[ J \dot{y}(t) = \frac{-G\Gamma}{c} \big( y(t) + \frac{c}{G} v(t-\omega) + D_\phi(t) \big) - F_e(y(t)), \]
and since $v(t-\omega), D_\phi(t), F_e(y(t))$ are bounded on $[K\omega,(K+1)\omega]$, existence of a solution $y$ with \( y \in W^{1,\infty}([K\omega, (K+1)\omega],\R) \) follows similar to Step~3. 

\textit{Step 6 (Global extension and boundedness).}
By induction, the solution extends to all intervals \( [0, N\omega] \), for any \( N \in \mathbb{N} \). Since $N$ can be arbitrarily large, the solution is global (i.e., exists for all \( t \ge 0 \)). The global boundedness is ensured by $|e(t)| \le \eps$ for $t\ge 0$ as shown in Steps~1 and~4. This also guarantees boundedness of all other signals $y, v, u, I, \phi$.
\end{proof}

\begin{rem}\label{rem:actual-error}
Theorem~\ref{Thm:funnel}~(iii) concerns the corrected error. For the actual tracking error,
the first equation in~\eqref{eq:FC} yields, for all $t\ge0$,
\[
  y(t)-y_{\rm ref}(t) = \psi(t)\,e(t+\omega) - I(t+\omega),
\]
so that $|y(t)-y_{\rm ref}(t)|\le\psi(t)+|I(t+\omega)|$. The correction term $I$ is driven by the
variation of the control signal over one round-trip time $2\omega$ of the wave, cf.\ the third
equation in~\eqref{eq:FC}. For a constant reference, $v$ settles and $I(t)\to0$, so that
the actual error is eventually bounded by $\psi$. For a time-varying reference, a quasi-stationary
consideration of the bit dynamics in~\eqref{eq:closed loop} shows that
$I(t)\approx y(t+\omega)-y(t-\omega)$ and hence $y(t)\approx y_{\rm ref}(t-2\omega)$: the
controller tracks the reference delayed by the round-trip time, and the actual error is small
whenever $y_{\rm ref}$ varies little over intervals of length $2\omega$, e.g.\
$|y_{\rm ref}(t+\omega)-y_{\rm ref}(t-\omega)|\le 2\omega\|\dot y_{\rm ref}\|_\infty$. This
reflects a structural limitation: the effect of the input on the bit is observed at the top of
the string only after $2\omega$, and no output feedback can react faster.
\end{rem}

\section{Simulations}\label{Sec:Sim}

We demonstrate the performance of the funnel controller~\eqref{eq:FC} applied to the drill string system~\eqref{drillstring} with zero initial conditions, that is  $\phi(\xi,0)=\phi_0(\xi)=0$ and $\phi_t(\xi,0)=v_0(\xi)=0$. The control input $u(t)$ is applied at the top boundary $\xi = 0$ via the boundary condition $G \phi_\xi(0,t)=u(t)$, and the output $y(t) = \phi_t(\ell,t)=$ is defined at the bottom boundary. The dynamics include distributed damping $F_d(\xi,v)$, which we choose to be a tangent inverse function specified in Table~\ref{table}, and nonlinear boundary dissipation at $\xi = \ell$ of the form
\[
\textbf{v}\mapsto F_e(\textbf{v}) = \frac{A \textbf{v}}{\sqrt{\textbf{v}^2 + \varepsilon^2}},
\]
which is continuous, monotone and bounded, and is a smooth single-valued approximation of the set-valued Coulomb relation $A\operatorname{sgn}(v)$; in particular, $F_e$ satisfies Assumption~\ref{assump for the controller}\,(i). The goal is to ensure, for the constant output reference $y_{\rm ref}(t) = 5$ (for $\ell=1$) and the time-varying output reference $y_{\rm ref}=5+\sin (t)$ (for $\ell=10$), the evolution of the corrected tracking error $w(t) = y(t) - y_{\rm ref}(t) + I(t+\omega)$ within the performance funnel defined by the funnel boundary function $\psi(t) = 8 e^{-t} + 0.1$ (for $\ell=1$) and $\psi(t) = 8 e^{-0.05t} + 0.1$ (for $\ell=10$), where $\omega = \ell/c = \ell\sqrt{\frac{\rho}{G}}$. The controller parameters are $\alpha=\beta=1$ in all simulations, which are the values prescribed by Theorem \ref{Thm:funnel} for $G=c=\Gamma=J=1$.  The simulation parameters are summarized in Table~\ref{table}. 

The funnel controller~\eqref{eq:FC} includes the transient shaping term $\left(\hat{v}-\frac{ke(0)}{1-e(0)^2}\right)p(t)$ which enforces a smooth initialization of $v$ at $t=0$. Since the initial data are zero, compatibility requires $\hat v=G\phi_0'(0)-\frac Gcv_0(0)=0$. We choose
\[
p(t) =
\begin{cases}
\exp\!\left(1-\dfrac{1}{\tau(1-\tau)}\right), & \text{if } 0 < t < T_{\text{shape}}, \\[8pt]
0, & \text{otherwise,}
\end{cases} \qquad
\text{where } \tau = \dfrac{t}{T_{\text{shape}}},\quad T_{\text{shape}}=0.5s,
\]
which is smooth, has compact support and satisfies $p(0)=1$ as required in \eqref{eq:FC}.
\begin{table}[!ht]
\centering
\label{tab:params}
\centering
\label{tab:params}
\begin{tabular}{lccc}
\toprule
\textbf{Parameter} & \textbf{Simulation 1} & \textbf{Simulation 2} & \textbf{Simulation 3} \\
\midrule
$\ell$ & $1.0$ & $1.0$ & $10.0$ \\
$N$ & $51$ & $500$ & $5000$ \\
$d\xi = \ell/(N-1)$ & $0.02$ & $0.002$ & $0.02$ \\
$\rho(\xi)$ & $1.0$ & $1.0$ & $1.0$ \\
$G(\xi)$ & $1.0$ & $1.0$ & $1.0$ \\
$F_d(\xi,v)$ & $0.3\arctan(v)$ & $0.3\arctan(v)$ & $0.3\arctan(v)$ \\
$A$ & $1.0$ & $1.0$ & $1.0$ \\
$\varepsilon$ & $0.001$ & $0.001$ & $0.001$ \\
$\Gamma$ & $1.0$ & $1.0$ & $1.0$ \\
$J$ & $1.0$ & $1.0$ & $1.0$ \\
$\alpha$ & $1.0$ & $1.0$ & $1.0$ \\
$\beta$ & $1.0$ & $1.0$ & $1.0$ \\
$k$ & $1.0$ & $1.0$ & $5$ \\                  
$y_{\text{ref}}(t)$ & $5.0$ & $5.0$ & $5.0+\sin(t/5)$ \\
$\psi(t)$ & $8e^{-t}+0.1$ & $8e^{-t}+0.1$ & $8e^{-0.1t}+0.1$ \\
$\hat{v}$ & $0.0$ & $0.0$ & $0.0$ \\
$T_{\text{shape}}$ & $0.5$ & $0.5$ & $0.5$ \\
$t \in [0,T]$ & $T=10$ & $T=10$ & $T=100$ \\
\bottomrule
\end{tabular}
\caption{Simulation parameters across the three plots.}
\label{table}
\end{table}

For the simulation, the PDE in~\eqref{drillstring} is converted to a system of ODEs using a second-order finite difference scheme with $N=51$ and $N=500$ spatial points over the domain $[0, \ell]$ for $\ell=1$ (Simulations 1 and 2) and $N=5000$ for $\ell=10$ (Simulation 3).

The ODE system is simulated using \texttt{solve\_ivp} with method \texttt{RK23}. The delay terms $y(t - \omega)$, $I(t - \omega)$, and $\psi(t - \omega)$ are computed using linear interpolation over a history buffer populated with pre-initialized values. The control input $u(t)$ is realized via a ghost-point method~\cite{tseng2003ghost}, where $\phi_{-1}$ is the ghost point introduced outside the domain to impose the boundary condition through a finite-difference stencil:
\[
\phi_{-1} = \phi_0 - \frac{d\xi}{G} u(t),
\]
so that the top boundary incorporates the effect of the control input. 

The simulation results depicted in Figure~\ref{fig:imageN=150} show that the output $y(t)$ closely tracks the constant desired value $y_{\rm ref}(t) = 5.0$. After an initial transient, tracking is achieved with high precision. The transient is well-contained due to the wide initial funnel. The error remains strictly within the corrected funnel bounds throughout the simulation. The controller dynamically adjusts $I(t)$ to maintain feasibility even as $\psi(t)$ decreases rapidly over time. This confirms feasibility of the control scheme. Figure~\ref{fig:imageN=500} shows the same scenario computed on a refined mesh with $N=500$ grid points. The
output and the tracking error are visually indistinguishable from Figure~\ref{fig:imageN=150}, whereas the
high-frequency oscillations of the control input visible in Figure~\ref{fig:imageN=150} for $t\gtrsim 5$ disappear;
they are therefore a discretization artifact of the coarse mesh and not a property of the
controller.

\begin{figure}
    \centering
    \includegraphics[width=0.9\linewidth]{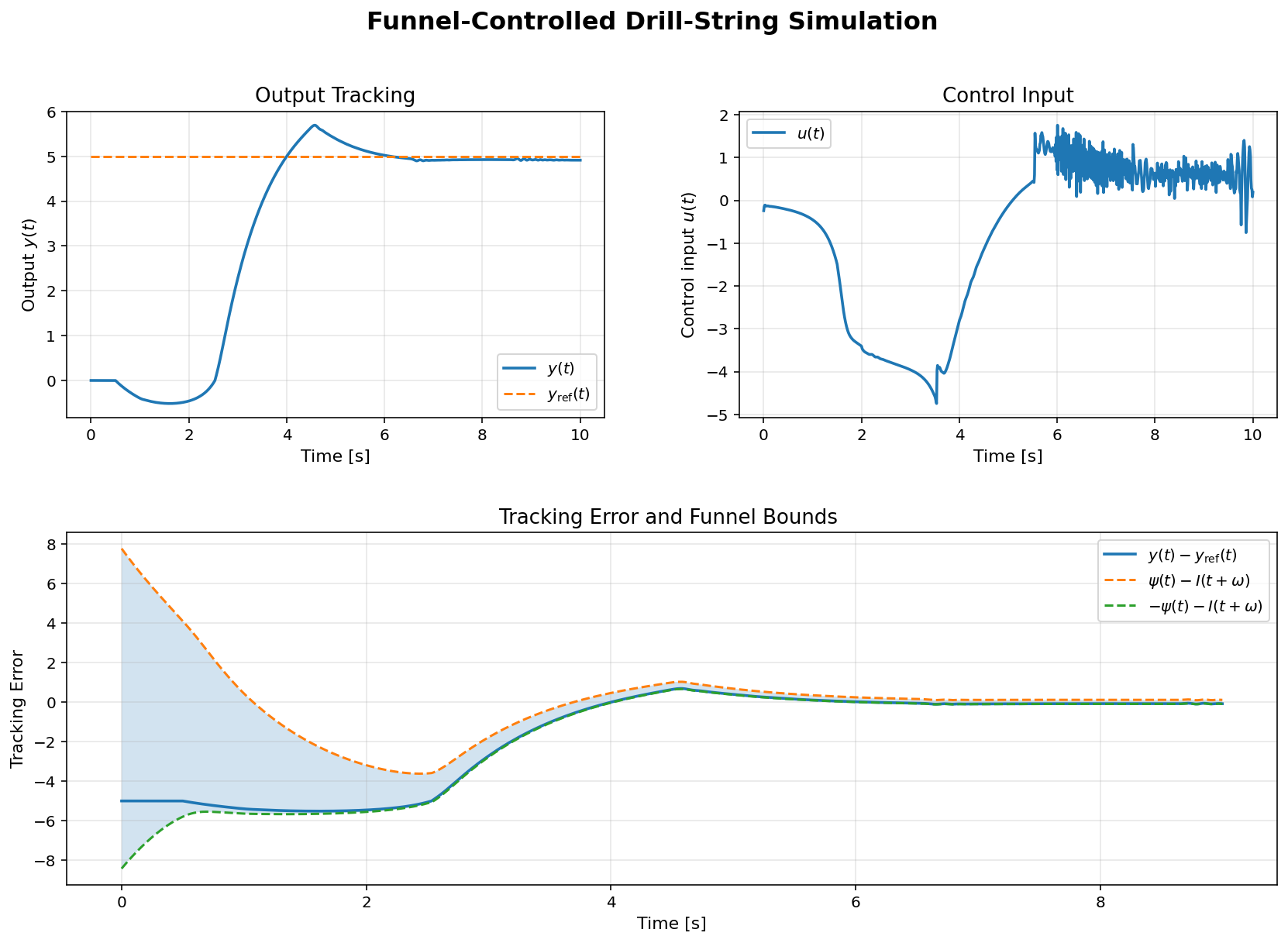}
    \caption{Simulation results for $\ell=1$, $y_{\rm ref}(t)=5$, $\psi(t)=8e^{-t}+0.1$, N=51.}
    \label{fig:imageN=150}
\end{figure}
\begin{figure}
    \centering
    \includegraphics[width=0.9\linewidth]{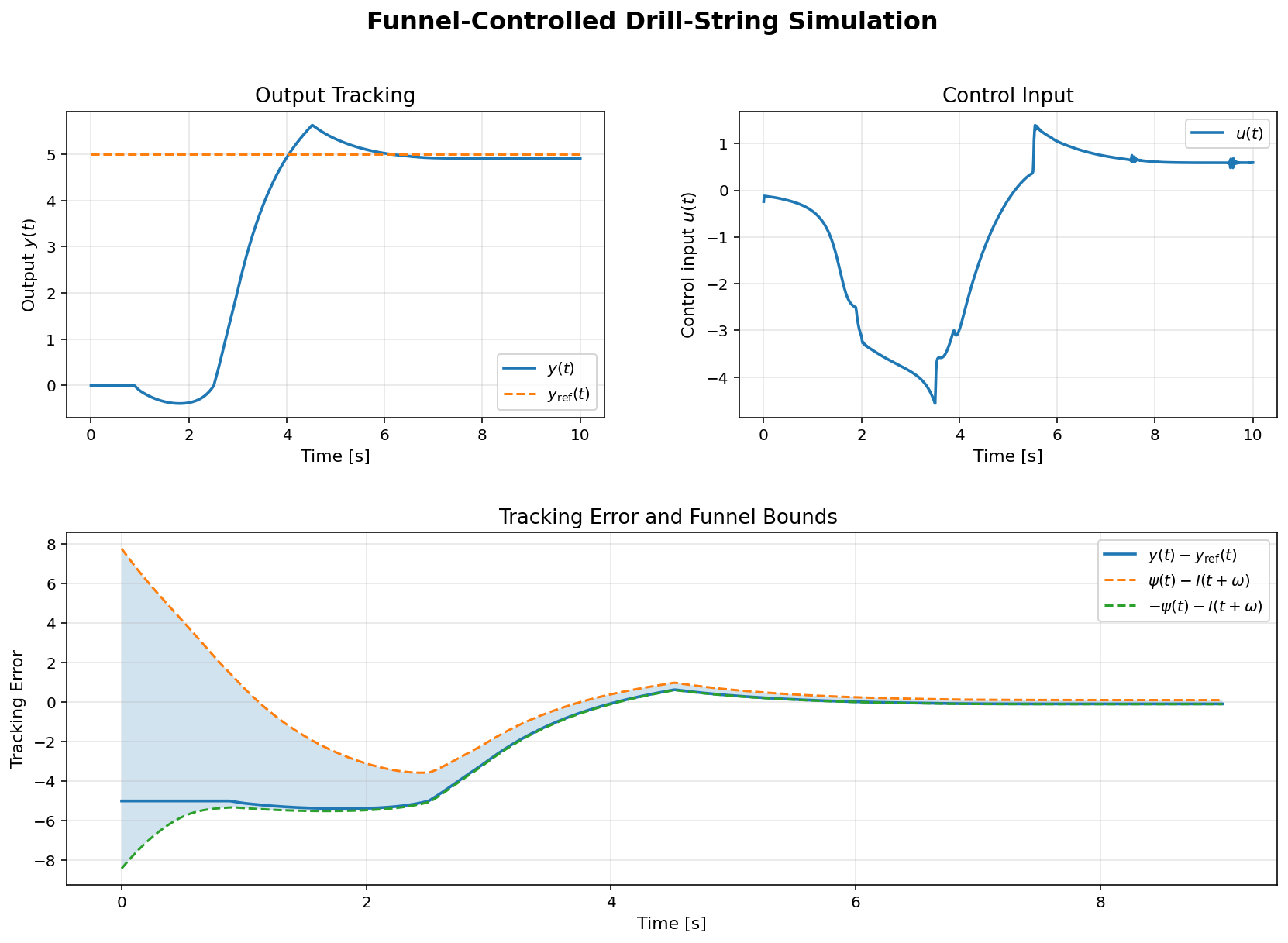}
    \caption{Simulation results for $\ell=1$, $y_{\rm ref}(t)=5$, $\psi(t)=8e^{-t}+0.1$, N=500.}
    \label{fig:imageN=500}
\end{figure}
\begin{figure}
    \centering
    \includegraphics[width=0.9\linewidth]{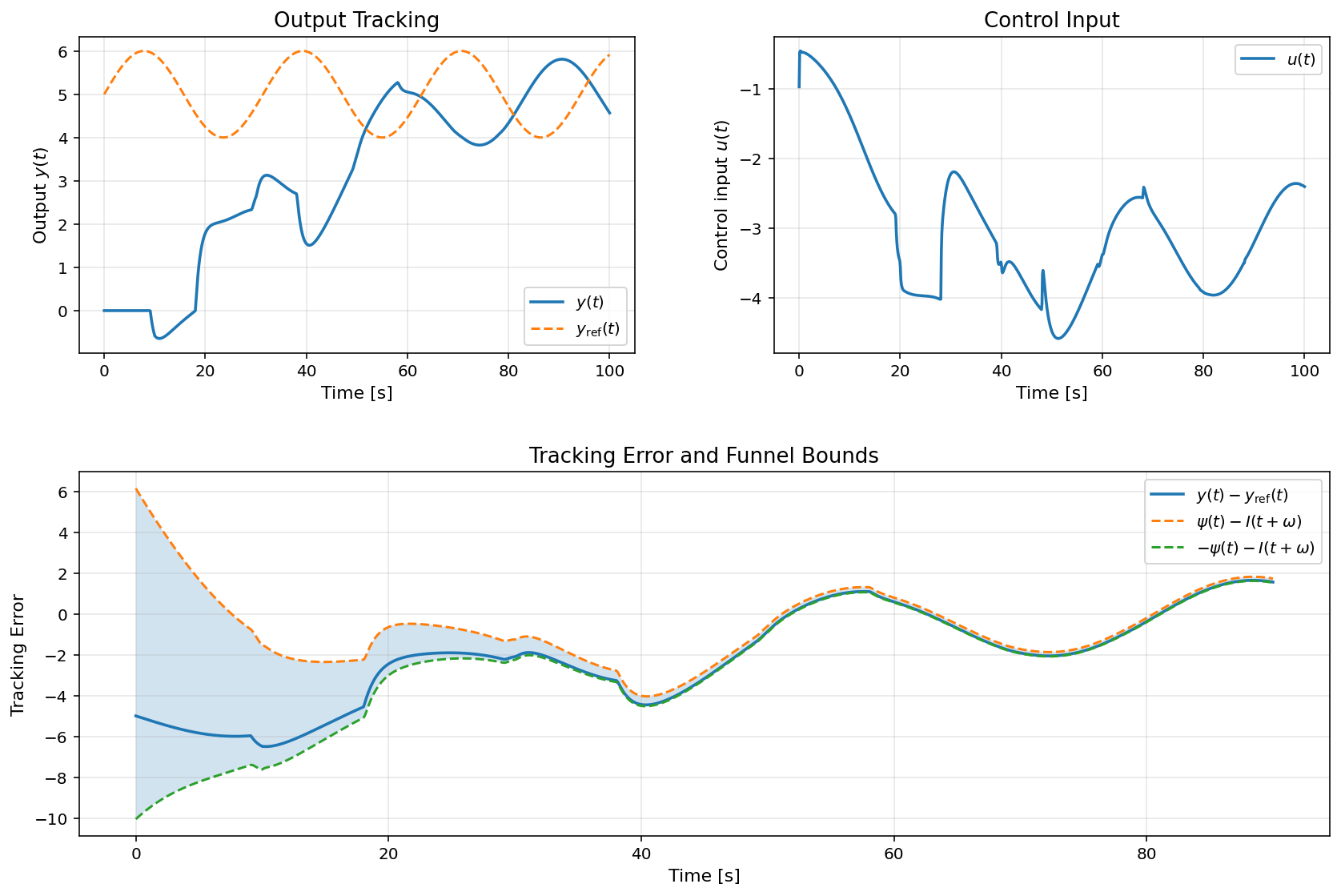}
    \caption{Simulation results for $\ell=10$, $y_{\rm ref}(t)=5 + \sin (t/5)$, $\psi(t)=8e^{-0.1t}+0.1$, N=5000.}
    \label{fig:placeholder}
\end{figure}

Figure \ref{fig:placeholder} depicts the plots for $\ell=10$, $\psi(t)=8e^{-0.1t}+0.1$ and $y_{\rm ref}(t)=5+\sin(t/5)$ over the time interval $[0,100]$, where the parameters used are $k=5$ and $N=5000$. All the parameters are listed in Table~\ref{table}. $y$ follows the reference with a delay of approximately $2\omega=20$, in accordance with Remark~\ref{rem:actual-error}, while the corrected error stays within the funnel.

\section{Conclusion}\label{Sec:Concl}
In this work, we studied the output tracking problem for a vertically driven drill string modeled by a nonlinear boundary-coupled PDE-ODE system, described using a non-homogeneous wave equation with spatially varying coefficients $\rho$ and $G$ as well as distributed and end-point damping. We studied the solvability of the system by first formulating it as an abstract  boundary value problem involving nonlinear set-valued operators on an appropriate  Hilbert space and then used the theory of maximal monotone operators to establish existence of generalized and strong solutions under the respective regularity of the input function. For the application of funnel control we restricted ourselves to the simpler case of constant physical coefficients $\rho$ and $G$ as well as single-valued damping functions in order to provide a rigorous feasibility proof for the proposed controller. The designed funnel controller incorporates delayed system states and an integral term, guaranteeing that the closed-loop system admits a global solution, the control input remains bounded, and the (possibly corrected) tracking error stays strictly within the prescribed performance funnel, as formalized in Theorem~\ref{Thm:funnel}. Simulation results validate the effectiveness of the controller in enforcing the drill bit angular velocity to follow a dynamically adjusted reference trajectory while respecting the funnel bounds.

Two questions are left for future work. First, the controller design and the
feasibility analysis rely on d'Alembert's representation for the constant-coefficient
wave equation; extending the delay-compensation mechanism to spatially varying
$\rho$ and $G$, for which the open-loop theory of Section~\ref{Sec:wellpose} already
applies, requires a different route to the delay dynamics of the bit velocity. Second,
the controller uses the wave traveling time $\omega$, and its robustness with respect
to a mismatch in this quantity remains to be analyzed.


\backmatter

\section*{Use of generative artificial intelligence}

OpenAI ChatGPT and Anthropic Claude were used to assist with language editing of the manuscript and with the development
and debugging of parts of the numerical implementation. All mathematical derivations, numerical results, and scientific conclusions are those of the authors. The authors independently verified all AI-assisted code and reviewed all AI-assisted content, and remain fully responsible for the content of the manuscript.

\section*{Funding}

This work was funded by the Deutsche Forschungsgemeinschaft (DFG, German Research Foundation) -- Project-ID 362536361.

\section*{Declarations}

\textbf{Competing Interests} The authors declare no competing interests.\\
\textbf{ Consent to Publish declaration:} not applicable.\\
\textbf{ Consent to Participate:} not applicable.\\
\textbf{Ethics declaration:} not applicable.

\bibliography{References}

@book{Bar2010,
author = {Barbu, Viorel},
year = {2010},
title = {Nonlinear Differential Equations of Monotone Type in Banach Spaces},
publisher={Springer},
address={New York, NY},
doi = {10.1007/978-1-4419-5542-5}
}

@article{BALANOV2003,
title = {Bifurcation analysis of a neutral delay differential equation modelling the torsional motion of a driven drill-string},
journal = {Chaos, Solitons \& Fractals},
volume = {15},
number = {2},
pages = {381-394},
year = {2003},
issn = {0960-0779},
doi = {10.1016/S0960-0779(02)00105-4},
author = {A.G. Balanov and N.B. Janson and P.V.E. McClintock and R.W. Tucker and C.H.T. Wang}
}

@article{CRUZNETO2023,
title = {Control of drill string torsional vibrations using optimal static output feedback},
journal = {Control Engineering Practice},
volume = {130},
pages = {105-366},
year = {2023},
issn = {0967-0661},
doi = {10.1016/j.conengprac.2022.105366},
author = {H.J. {Cruz Neto} and M.A. Trindade}
}

@article{AndAhm2003,
author = {Christoforou, Andreas and Yigit, Ahmet},
year = {2003},
pages = {1029-1045},
title = {Fully coupled vibrations of actively controlled drillstrings},
volume = {267},
journal = {Journal of Sound and Vibration},
doi = {10.1016/S0022-460X(03)00359-6}
}

@article{BerIlcRya2021,
title = {Funnel control of nonlinear systems},
journal ={Mathematics of Control, Signals and Systems},
volume = {33},
pages = {151-194},
year = {2021},
author = {Berger, Thomas and Ilchmann, Achim and Ryan, Eugene P},
doi = {10.1007/s00498-021-00277-z}
}

@article{Berger2018,
  title={Funnel control for nonlinear systems with known strict relative degree},
  author={Berger, Thomas and L{\^e}, Huy Ho{\`a}ng and Reis, Timo},
  journal={Automatica},
  volume={87},
  pages={345--357},
  year={2018},
  publisher={Elsevier},
  doi = {10.1016/j.automatica.2017.10.017}
}

@article{Ber2020,
title = {Funnel control in the presence of infinite-dimensional internal dynamics},
journal = {Systems \& Control Letters},
volume = {139},
pages = {104678},
year = {2020},
author = {Thomas Berger and Marc Puche and Felix L. Schwenninger},
doi = {10.1016/j.sysconle.2020.104678}
}

@article{SagMegKrsRou2013,
  author  = {Sagert, C. and {Di Meglio}, F. and Krsti{\'c}, M. and Rouchon, P.},
  title   = {Backstepping and Flatness Approaches for Stabilization of the Stick-Slip Phenomenon for Drilling},
  journal = {IFAC Proceedings Volumes},
  volume  = {46},
  number  = {2},
  pages   = {779--784},
  year    = {2013},
  doi     = {10.3182/20130204-3-FR-2033.00126}
}

@article{IlcRyaSan2002,
author = {Ilchmann, Achim and Ryan, E.P. and Sangwin, Christopher},
journal = {ESAIM: Control, Optimisation and Calculus of Variations},
pages = {471--493},
volume = {7},
year = {2002},
title = {Tracking with prescribed transient behaviour},
doi = {10.1051/cocv:2002064}
}

@book{evans2022partial,
  title={Partial differential equations},
  author={Evans, Lawrence C},
  volume={19},
  year={2022},
  publisher={American Mathematical Society},
}

@book{folland2020introduction,
  title={Introduction to partial differential equations},
  author={Folland, Gerald B},
  year={2020},
  publisher={Princeton university press},
}

@article{IlcTre2004,
title = {Input constrained funnel control with applications to chemical reactor models},
journal = {Systems \& Control Letters},
volume = {53},
number = {5},
pages = {361-375},
year = {2004},
author = {Achim Ilchmann and Stephan Trenn},
doi = {10.1016/j.sysconle.2004.05.014}
}

@article{BerRei2014,
title = {Zero dynamics and funnel control for linear electrical circuits},
journal = {Journal of the Franklin Institute},
volume = {351},
number = {11},
pages = {5099-5132},
year = {2014},
author = {Thomas Berger and Timo Reis},
doi = {10.1016/j.jfranklin.2014.08.006}
}

@INPROCEEDINGS{Hac2014,
  author={Hackl, C. M.},
  booktitle={2014 IEEE Conference on Control Applications (CCA)}, 
  title={Funnel control for wind turbine systems}, 
  year={2014},
  volume={},
  number={},
  pages={1377-1382},
  doi = {10.1109/CCA.2014.6981516}
  }

@article{tian2020research,
  author  = {Tian, Jialin and Wei, Lai and Yang, Lin and Dai, Liming and Zhang, Tangjia and Liu, He},
  title   = {Research and Experimental Analysis of Drill String Dynamics Characteristics and Stick-Slip Reduction Mechanism},
  journal = {Journal of Mechanical Science and Technology},
  volume  = {34},
  number  = {3},
  pages   = {977--986},
  year    = {2020},
  doi     = {10.1007/s12206-020-0201-9}}

@book{loomis1968advanced,
  title={Advanced Calculus},
  author={Loomis, Lynn Harold and Sternberg, Shlomo},
  year={1968},
  publisher={World Scientific},
}

@article{TerAndVin2020,
  author  = {Terrand-Jeanne, Alexandre and Andrieu, Vincent and Tayakout-Fayolle, M{\'e}laz and {Dos Santos Martins}, Val{\'e}rie},
  title   = {Regulation of Inhomogeneous Drilling Model With a {P-I} Controller},
  journal = {IEEE Transactions on Automatic Control},
  volume  = {65},
  number  = {1},
  pages   = {58--71},
  year    = {2020},
  doi     = {10.1109/TAC.2019.2907792}
}

@article{Chitour2023,
  title={Lyapunov functions for linear damped wave equations in one-dimensional space with dynamic boundary conditions},
  author={Yacine Chitour and Hoai-Minh Nguyen and Christophe Roman},
  journal={Automatica},
  year={2023},
  volume={167},
  pages={111754},
}

@article{athanasiou2020simulation,
  author  = {Athanasiou, Panagiotis and Hadi, Yaser},
  title   = {Simulation of Oil Well Drilling System Using Distributed--Lumped Modelling Technique},
  journal = {Modelling},
  volume  = {1},
  number  = {2},
  pages   = {175--197},
  year    = {2020},
  doi     = {10.3390/modelling1020011}
}

@article{sharma2024review,
  author  = {Sharma, Aditya and Abid, Khizar and Srivastava, Saket and Velasquez, Andres Felipe Baena and Teodoriu, Catalin},
  title   = {A Review of Torsional Vibration Mitigation Techniques Using Active Control and Machine Learning Strategies},
  journal = {Petroleum},
  volume  = {10},
  number  = {3},
  pages   = {411--426},
  year    = {2024},
  doi     = {10.1016/j.petlm.2023.09.007}
}

@article{moharrami2021nonlinear,
  author  = {Moharrami, Mohammad Javad and de Arruda Martins, Cl{\'o}vis and Shiri, Hodjat},
  title   = {Nonlinear Integrated Dynamic Analysis of Drill Strings Under Stick-Slip Vibration},
  journal = {Applied Ocean Research},
  volume  = {108},
  pages   = {102521},
  year    = {2021},
  doi     = {10.1016/j.apor.2020.102521}
}

@article{deutscher2018output,
title = {Output feedback control of general linear heterodirectional hyperbolic {ODE}–{PDE}–{ODE} systems},
journal = {Automatica},
volume = {95},
pages = {472-480},
year = {2018},
doi = {10.1016/j.automatica.2018.06.021},
author = {Joachim Deutscher and Nicole Gehring and Richard Kern},
}

@article{Deutscher03102019,
author = {J. Deutscher and N. Gehring and R. Kern},
title = {Output feedback control of general linear heterodirectional hyperbolic {PDE}-{ODE} systems with spatially-varying coefficients},
journal = {International Journal of Control},
volume = {92},
number = {10},
pages = {2274--2290},
year = {2019},
publisher = {Taylor \& Francis},
doi = {10.1080/00207179.2018.1436770}

}

@article{BergBika25,
  title={Prescribed performance control of uncertain higher-order nonlinear systems in the presence of delays},
  author={Thimas Berger and Lampros N. Bikas and Jan Hachmeister and George A. Rovithakis},
  journal={arXiv preprint arXiv:2509.08601},
  year={2025}
}

@incollection{Ilchmann2012DecentralizedTO,
  author    = {Ilchmann, Achim},
  title     = {Decentralized tracking of interconnected systems},
  booktitle = {Mathematical System Theory -- Festschrift in Honor of Uwe Helmke on the Occasion of his Sixtieth Birthday},
  editor    = {H{\"u}per, Knut and Trumpf, Jochen},
  publisher = {CreateSpace},
  year      = {2013},
  pages     = {229--245},
  isbn      = {9781470044008}
}

@article{BergPuch22,
	Author = {Berger, Thomas and Puche, Marc and Schwenninger, Felix L.},
	Title = {Funnel control for a moving water tank},
    journal = {Automatica},
    volume = 135,
    pages = {109999},
	Year = 2022,
    doi = {10.1016/j.automatica.2021.109999}
    
}

@article{BergIlch25, 
    title={Funnel control - a survey}, 
    author={Thomas Berger and Achim Ilchmann and Eugene P. Ryan},
    year={2025},
    journal={Annual Reviews in Control},
    volume = 60, 
    pages = {101024},
    doi = {10.1016/j.arcontrol.2025.101024}
}

@BOOK{Hack17,
   AUTHOR    = {Hackl, Christoph M.},
   YEAR      = 2017,
   TITLE     = {Non-identifier Based Adaptive Control in Mechatronics--Theory and Application},
   PUBLISHER = {Springer},
   Address   = {Cham, Switzerland},
   Series    = {Lecture Notes in Control and Information Sciences},
   Volume    = 466,
}

@article{BergDrue21,
	Author = {Berger, Thomas and Drücker, Svenja and Lanza, Lukas and Reis, Timo and Seifried, Robert},
	Journal = {Nonlinear Dynamics},
	Title = {Tracking control for underactuated non-minimum phase multibody systems},
	Year = 2021,
    Volume = 104,
    Issue  = 4,
    Pages = {3671--3699},
    doi = {10.1007/s11071-021-06458-4}
}

@article{DrueLanz24,
	Author = {Drücker, Svenja and Lanza, Lukas and Berger, Thomas and Reis, Timo and Seifried, Robert},
	Journal = {Multibody System Dynamics},
	Title = {Experimental validation for the combination of funnel control with a feedforward control strategy},
	Year = 2024,
    Volume = 63,
    Pages = {105--123},
    doi = {10.1007/s11044-024-09976-2}
}

@article{Rabinowicz1956,
  author  = {Rabinowicz, Ernest},
  title   = {Stick and Slip},
  journal = {Scientific American},
  volume  = {194},
  number  = {5},
  pages   = {109--119},
  year    = {1956},
  doi     = {10.1038/scientificamerican0556-109}
}

@article{Rockafellar1970,
  author  = {R. Tyrrell Rockafellar},
  title   = {On the maximal monotonicity of subdifferential mappings},
  journal = {Pacific Journal of Mathematics},
  volume  = {33},
  year    = {1970},
  pages   = {209--216},
  doi = {10.2140/pjm.1970.33.209
}
}

@article{Tolstonogov2021,
  author  = {A. A. Tolstonogov},
  title   = {Maximal Monotonicity of a {N}emytskii Operator},
  journal = {Funct. Anal. Appl.},
  year    = {2021},
  volume  = {55},
  number  = {3},
  pages   = {217--225},
  doi     = {10.1134/S0016266321030047}
}

@article{Evans1977,
  author  = {Evans, Lawrence C.},
  title   = {Nonlinear evolution equations in an arbitrary {B}anach space},
  journal = {Israel J. Math.},
  volume  = {26},
  number  = {1},
  year    = {1977},
  pages   = {1--42},
  doi     = {10.1007/BF03007654}
}

@article{tseng2003ghost,
  title={A ghost-cell immersed boundary method for flow in complex geometry},
  author={Tseng, Yu-Heng and Ferziger, Joel H},
  journal={Journal of computational physics},
  volume={192},
  number={2},
  pages={593--623},
  year={2003},
  doi = {10.1016/j.jcp.2003.07.024}
}

@article{Fattorini1968,
  author  = {H. O. Fattorini},
  title   = {Boundary Control Systems},
  journal = {SIAM Journal on Control},
  volume  = {6},
  number  = {3},
  pages   = {349--385},
  year    = {1968},
  doi     = {10.1137/0306025}
}

@incollection{lhachemi:hal-05038849,
  TITLE = {{Nonlinear Control for Infinite-Dimensional Systems}},
  AUTHOR = {Lhachemi, Hugo and Prieur, Christophe},
  BOOKTITLE = {{Reference Module in Materials Science and Materials Engineering}},
  PUBLISHER = {{Elsevier}},
  YEAR = {2025},
  DOI = {10.1016/B978-0-443-14081-5.00150-1},
  HAL_ID = {hal-05038849},
  HAL_VERSION = {v1},
}

@article{favini2001nonlinear,
  title={Nonlinear boundary conditions for nonlinear second order differential operators on C [0, 1]},
  author={Favini, ANGELO and Goldstein, G Ruiz and Goldstein, JEROME A and Romanelli, Silvia},
  journal={Archiv der Mathematik},
  volume={76},
  number={5},
  pages={391--400},
  year={2001},
  publisher={Springer},
  doi = {10.1007/PL00000449}
  
}

@article{avalos2001well,
  title={Well-posedness of a structural acoustics control model with point observation of the pressure},
  author={Avalos, George and Lasiecka, Irena and Rebarber, Richard},
  journal={Journal of Differential Equations},
  volume={173},
  number={1},
  pages={40--78},
  year={2001},
  publisher={Elsevier},
  doi = {10.1006/jdeq.2000.3938}
}

@article{graber2012existence,
  title={Existence and asymptotic behavior of the wave equation with dynamic boundary conditions},
  author={Graber, Philip Jameson and Said-Houari, Belkacem},
  journal={Applied Mathematics \& Optimization},
  volume={66},
  number={1},
  pages={81--122},
  year={2012},
  publisher={Springer},
  doi={10.1007/s00245-012-9165-1}
}

@book{BauschkeCombettes2017,
  author    = {Bauschke, Heinz H. and Combettes, Patrick L.},
  title     = {Convex Analysis and Monotone Operator Theory in Hilbert Spaces},
  edition   = {2},
  series    = {CMS Books in Mathematics},
  publisher = {Springer},
  address   = {Cham},
  year      = {2017},
  doi       = {10.1007/978-3-319-48311-5}
}

@book{Pazy83,
    author    = {Pazy, Amnon},
  title     = {Semigroups of Linear Operators and Applications to Partial Differential Equations},
  series    = {Applied Mathematical Sciences},
  volume    = 44,
  publisher = {Springer},
  address   = {New York},
  year      = {1983},
  doi       = {10.1007/978-1-4612-5561-1},
}

@book{Walter1998,
  author    = {Wolfgang Walter},
  title     = {Ordinary Differential Equations},
  series    = {Graduate Texts in Mathematics},
  volume    = {182},
  publisher = {Springer-Verlag},
  address   = {New York, NY, USA},
  year      = {1998}
}

@book{Filippov1988DifferentialEW,
  title={Differential Equations with Discontinuous Righthand Sides},
  author={Aleksej F. Filippov},
  series={Mathematics and Its Applications},
  volume=18,
  year={1988},
  publisher={Kluwer},
  address={Dordrecht},
  doi={10.1007/978-94-015-7793-9}
}

@book{Leine2004,
author="Leine, Remco I. and Nijmeijer, Henk",
title="Dynamics and Bifurcations of Non-Smooth Mechanical Systems",
year="2004",
publisher="Springer",
address="Berlin, Heidelberg",
isbn="978-3-540-44398-8",
doi="10.1007/978-3-540-44398-8",
}

@article{10.1115/1.1452745,
    author = {Leine, R. I. and van Campen, D. H. and Keultjes, W. J. G.},
    title = {Stick-slip Whirl Interaction in Drillstring Dynamics},
    journal = {Journal of Vibration and Acoustics},
    volume = {124},
    number = {2},
    pages = {209--220},
    year = {2002},
    doi = {10.1115/1.1452745},
}

\end{document}